\documentclass[12pt, a4paper]{article}
\usepackage[english]{babel}
\usepackage{amsfonts, amssymb, euscript, amsmath, mathtools, amsthm, comment}
\usepackage{authblk}
\usepackage{graphicx, wrapfig, multicol, capt-of, subcaption}
\usepackage[colorinlistoftodos]{todonotes}
\usepackage{hyperref}
\hypersetup{
    colorlinks=true,
    linkcolor=blue,
    citecolor=cyan,
    filecolor=magenta,      
    urlcolor=cyan,
    pdftitle={Overleaf Example},
    pdfpagemode=FullScreen,
    }
\DeclareGraphicsExtensions{.pdf,.png,.jpg}
\usepackage[utf8]{inputenc}
\usepackage[T1]{fontenc}
\usepackage{indentfirst}
\title{Bifurcations and Structural Stability of Generic PC-HC Families \thanks{This article is an output of a research project (HSE-BR-2025-84) implemented as part of the Basic Research Program at HSE University.}}
\author{Alexey Dorovskiy \thanks{HSE University, Faculty of Mathematics}}
\date{}

\begin{document}

\newcounter{theorem}[section]
\def\thetheorem{\arabic{section}.\arabic{theorem}}
\newcounter{definition}[section]
\def\thedefinition{\arabic{section}.\arabic{definition}}
\newcounter{proposition}[section]
\def\theproposition{\arabic{section}.\arabic{proposition}}
\newcounter{lemma}[section]
\def\thelemma{\arabic{section}.\arabic{lemma}}

\newenvironment{Theorem}{\refstepcounter{theorem}\par {\textbf{Theorem \thetheorem.}} \itshape}{}
\newcommand{\thm}{\par \refstepcounter{theorem} \textbf{Theorem \thetheorem.}}
\newcommand{\defi}{\par \refstepcounter{definition} \textbf{Definition \thedefinition. }}
\newcommand{\prop}{\par \refstepcounter{proposition} \textbf{Proposition \theproposition. }}
\newcommand{\lemma}{\par \refstepcounter{lemma} \textbf{Lemma \thelemma. }}

\maketitle
\begin{abstract}
In this paper the structural stability of generic families of vector fields of the PC-HC class on the two-dimensional sphere $S^2$ is proved. A classification of these families up to moderate equivalence in neighborhoods of their large bifurcation supports is presented, based on such invariants as the configuration and the characteristic set. The realization lemma is proved. Furthermore, bifurcation diagrams for the considered class of families are constructed.
\end{abstract}

\noindent \textbf{Key words:} structural stability, vector fields on the two-sphere, large bifurcation support, sparkling saddle connections, bifurcation diagrams.

\vfill
\pagebreak

\tableofcontents

\pagebreak

\section{Introduction}

\subsection{Glocal Bifurcations}

The classical bifurcation theory on the real plane consists of two main parts: the first studies local bifurcations occurring in neighborhoods of singular points, and the second studies semi-local bifurcations that may occur in neighborhoods of polycycles and parabolic cycles. More recently, a new direction has been developed, studying global bifurcations on the two-dimensional sphere $S^2$. It is based on a set of conjectures put forward by V.\,I.~Arnold in 1985~(\cite{Arnold's conjectures}), as well as on the work of Sotomayor~\cite{Sotomayor}. The main distinction of this new direction from classical theory arises from the appearance of the effect of sparkling saddle connections, which was discovered in~\cite{sparkling saddle connections}.

The main problems of the new theory are:
\begin{enumerate}
	\item To study the structural stability of finite-parameter families of vector fields on the two-dimensional sphere.
	\item To classify the bifurcations occurring in these families.
\end{enumerate}

Sotomayor described all the degeneracies that can appear in one-parameter families (unfoldings of so-called {\itshape quasi-generic} vector fields). He identified six types of bifurcations and initiated the study of their structural stability. This investigation was completed in~\cite{one-parameter family proof} and~\cite{one-parameter family proof 2}, where a complete classification of bifurcations of quasi-generic families was also given.

One of the basic conjectures of the theory of global bifurcations turned out to be false. It assumed that all finite-parameter families of vector fields in general position are structurally stable. An example of a structurally unstable family in general position with three parameters was constructed in~\cite{three-parameter family counterexample}. The instability arises due to sparkling saddle connections that accumulate in a neighborhood of a polycycle. Taking this result into account, the new problems of the theory of {\itshape glocal} bifurcations ({\itshape global} in the phase variable and {\itshape local} in the parameter) can be formulated as follows:
\begin{enumerate}
	\item To study the structural stability (or instability) of generic two-parameter families of vector fields on the two-dimensional sphere with respect to weak equivalence.
	\item To classify the bifurcations in these families.
	\item To find the ``boundary'' in the space of three-parameter families between stable and unstable families.
\end{enumerate}

One possible approach to the first problem is to partition the set of all generic families into classes and to prove that the families of each class are structurally stable (or unstable). Our goal is precisely the study of one such class: PC-HC families.

Throughout this paper, the term {\itshape vector fields} will refer to infinitely smooth vector fields on $S^2$ with isolated singular points. All homeomorphisms of the two-dimensional sphere $S^2$ will be assumed to be orientation-preserving.

\subsection{PC-HC Families}
In this paper we study glocal bifurcations in a broad class of generic two-parameter families of vector fields on the sphere with the following degeneracy: a semi-stable limit cycle $\gamma$ and a saddle-node $N$ with a homoclinic trajectory. Such families form the PC-HC class (from Parabolic Cycle---Homoclinic Curve); a precise definition will be given in Section~\ref{condition for a PC-HC family to be generic}. The main results of the paper are as follows:
\begin{itemize}
    \item a classification of the large bifurcation supports of PC-HC families is given: two families of the same class undergo the same bifurcations (up to minor details)
    \item the bifurcations occurring in generic PC-HC families are described
    \item the structural stability of generic PC-HC families is proved
\end{itemize}

Figure~\ref{PC-HC (A) vector field} shows the most interesting part of the phase portrait of the unperturbed vector field of the PC-HC class. Without loss of generality, we may assume that the nonzero eigenvalue of the saddle-node is positive; otherwise we reverse time. Thus, the saddle-node $N$, which in generic families of the PC-HC class has multiplicity 2, has three separatrices: two stable---$\beta_1$ and $\beta_2$---and one unstable---$\gamma_h$ (the index $h$ comes from the word homoclinic), which returns into the interior of the parabolic sector of the saddle-node and, together with the saddle-node, forms a closed curve
$$
\overline{\gamma}_h = \gamma_h \cup N.
$$

We now proceed to a more detailed description of the results.

\subsection{Configurations}

Consider the set $C(v)$ for a field $v$ of the PC-HC class:
$$
C(v) = \overline{\gamma}_h \cup \gamma \cup \beta_1,
$$
where $\beta_1$ is the stable separatrix of the saddle-node $N$ lying in the annulus bounded by the curves $\overline{\gamma}_h$ and $\gamma$.
The topology of this set, as well as the behavior of the trajectories of the field $v$ in its neighborhood, will be described by means of the following construction.

\defi We call the {\itshape configuration} of the field $v$ the word $a_1a_2a_3$ of three symbols in the alphabet $\{0, 1\}$ according to the following conditions:
\begin{enumerate}
    \item[i)] $a_1=1$, if the curves $\overline{\gamma}_h \cup \gamma$ have consistent time orientations, \\ $a_1=0$, if their orientations are opposite;
    
    \item[ii)] $a_2=1$, if $\alpha(\beta_1) = \gamma$ (i.e., $\gamma$ is the $\alpha$-limit set of the separatrix $\beta_1$), \\ $a_2=0$, if $\alpha(\beta_1) \ne \gamma$;
    
    \item[iii)] $a_3=1$, if the parabolic limit cycle $\gamma$ is repelling from the side of the curve $\overline{\gamma}_h$, \\ $a_3=0$, if the parabolic cycle is attracting from the side of the curve $\overline{\gamma}_h$.
\end{enumerate}

\begin{figure}
\center{\includegraphics[scale=0.77]{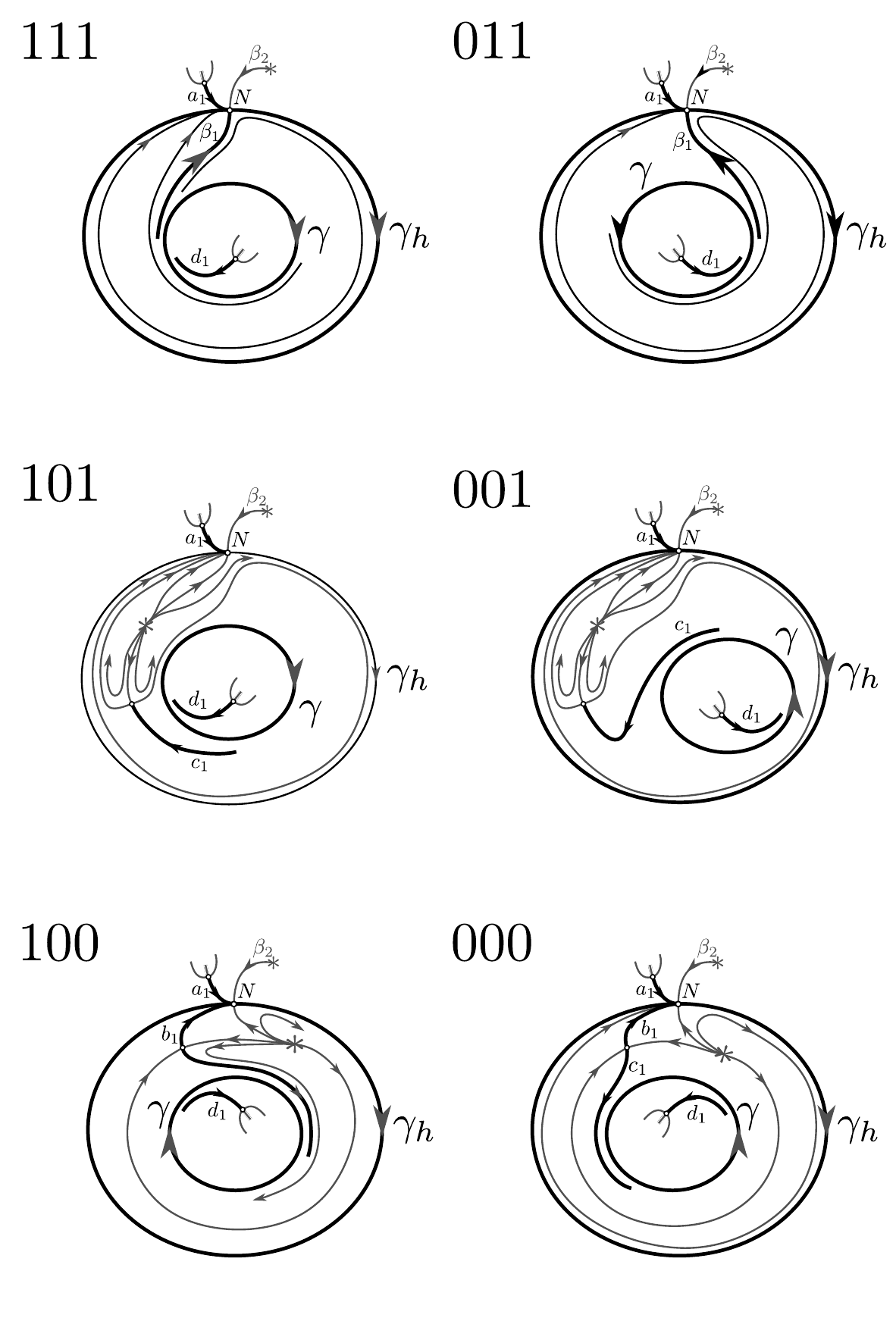}}
\caption{Examples of phase portraits of unperturbed fields from families of different configurations of the PC-HC class.
The asterisks denote hyperbolic sources of the field.
The bold curves highlight the trajectories belonging to the large bifurcation supports, whose definition is recalled later.
}
\label{PC-HC configurations (examples)}
\end{figure}

\lemma Every field $v$ of the PC-HC class has one of six configurations.

{\itshape Proof} of this lemma is simple enough to be given in the introduction. Indeed, the explicit form of the set $C(v)$ shows that its topology is determined by two properties, corresponding to the letters $a_1$ and $a_2$ in the configuration. The field $v$ can be defined in a small neighborhood of the set $C(v)$ in only two non-equivalent ways in the sense of orbital topological equivalence: the field $v$ is determined by whether the parabolic cycle $\gamma$ is repelling or attracting from the side of the curve $\overline{\gamma}_h$.
Each of the three properties is binary, so the letters $a_i$ in the configuration word can take two values, 0 and 1. Any field $v$ of the PC-HC class in a small neighborhood of the set $C(v)$ is uniquely determined, up to orbital topological equivalence, by the three specified properties.

There are $2^3 = 8$ configuration words in total, but taking into account the fact that the separatrix $\beta_1$ cannot wind onto the limit cycle $\gamma$ if the latter is attracting from the side of the saddle-node $N$, we can conclude that two of the eight cases do not occur, i.e., there are no vector fields $v$ of the PC-HC class whose configurations are words with $a_2=1$ and $a_3=0$. The remaining six cases are realizable. Examples of unperturbed fields corresponding to each of these cases are given in Figure~\ref{PC-HC configurations (examples)}.\qed

\begin{figure}[t]
\center{\includegraphics[scale=0.58]{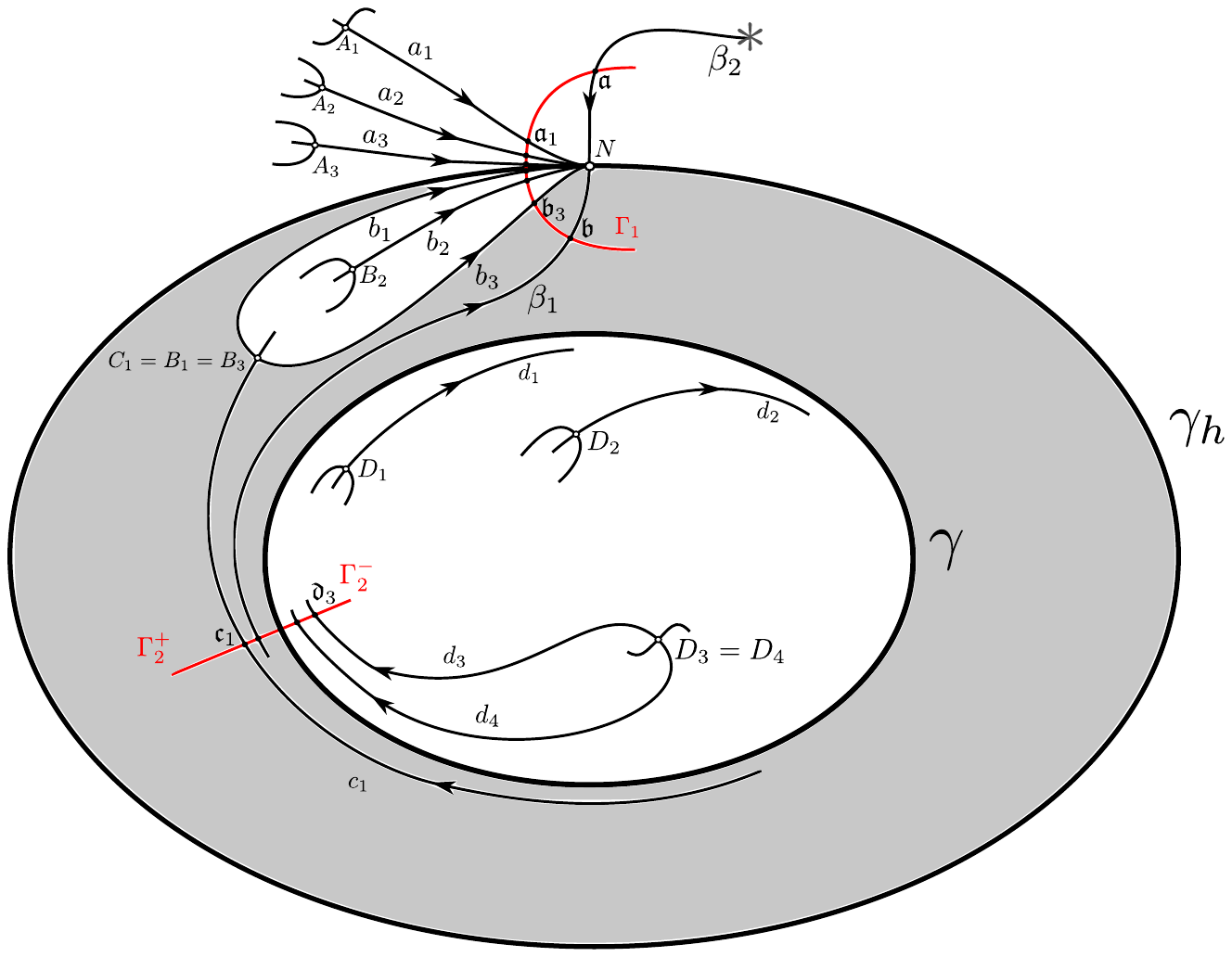}}
\caption{Phase portrait of the unperturbed field from a PC-HC family of configuration 111. The gray-shaded region belongs to the set $G$. The definition of the set $G$ is given in Section~\ref{subsection definition of large bifurcation support}.
}
\label{PC-HC (A) vector field}
\end{figure}

\subsection{Characteristic Sets and Liaison}
\label{subsection about characteristic sets}
Consider the phase portrait of the unperturbed field $v_0$ of a family of the PC-HC class. Let us draw a transversal curve $\Gamma_1$ in a neighborhood of the saddle-node $N$, intersecting all trajectories of its parabolic sector, as well as the stable separatrices $\beta_1$ and $\beta_2$ of the saddle-node $N$ at points $\mathfrak{a}$ and $\mathfrak{b}$ respectively. (See Figure~\ref{PC-HC (A) vector field}.) Let us introduce a smooth coordinate $\zeta$ on this curve so that the intersection of the transversal $\Gamma_1$ with the stable separatrix $\beta_1$ has coordinate $\zeta(\mathfrak{a}) = -1$, with the separatrix $\beta_2$ has coordinate $\zeta(\mathfrak{b})=1$, and with the homoclinic curve $\gamma_h$ has coordinate $\zeta=0$. Let the separatrices of saddles tending to the saddle-node $N$ in forward time intersect the transversal $\Gamma_1$ at points $\{\mathfrak{a}_i\}_i$ and $\{\mathfrak{b}_j\}_j$, whose coordinates satisfy $\zeta(\mathfrak{a}_i) < 0 \; \forall i$ and $\zeta(\mathfrak{b}_j) > 0 \; \forall j$.

Denote $\mathcal{L}_1 \vcentcolon = \{ \zeta(\mathfrak{a}_i) \mid 1 \le i \le l \}$ and $\mathcal{L}_2 \vcentcolon = \{ \zeta(\mathfrak{b}_j) \mid 1 \le j \le m \}$---two finite subsets on an interval.

Now draw a transversal section $\Gamma_2$ to the non-hyperbolic cycle $\gamma$. Introduce a smooth coordinate $x$ on it so that the intersection of this transversal with the limit cycle $\gamma$ has zero coordinate.
On the positive semi-interval $\Gamma_2^+$ of the transversal $\Gamma_2$ with the origin removed, one can introduce a time function $T^+$ mapping this interval to the positive half-axis, so that the Poincaré return map $\mathcal{P}$, defined on the semi-transversal $\Gamma_2^+$, takes the form of a unit shift in the new chart $T^+$:
$$
T^+(\mathcal{P}(x)) = T^+(x) + 1.
$$
An analogous chart $T^-$ can be introduced on the negative semi-transversal $\Gamma_2^-$. The existence of such time charts is proved in~\cite{one-parameter family proof 2}. The orbit space of the Poincaré map $\mathcal{P}$ on the curve $\Gamma_2$ can be parameterized by points of two semi-intervals $[b^{\pm}, \mathcal{P}(b^{\pm}))$, where $b^+$ and $b^-$ are two fixed points on the transversal $\Gamma_2$, chosen in a neighborhood of the origin $x=0$ so that $x(b^-) < 0 < x(b^+)$. Equivalently, this orbit space can be represented as two circles $S^1_{\pm} = \mathbb{R}_{\ge 0} / \mathbb{Z}$ (the circles $S^1_+$ and $S^1_-$ correspond to orbits lying in the regions $x>0$ and $x<0$ respectively).

Let the separatrices of saddles tending to the parabolic cycle in forward or backward time, as well as the separatrix $\beta_1$ if it tends to the cycle in backward time, intersect the semi-intervals $[b^{\pm}, \mathcal{P}(b^{\pm}))$ at points $\{\mathfrak{c}_r\}_r$ and $\{\mathfrak{d}_s\}_s$ (let the points $\{\mathfrak{c}_r\}_r$ correspond to separatrices lying in the topological annulus bounded by the curves $\gamma$ and $\overline{\gamma}_h$, and the points $\{\mathfrak{d}_s\}_s$ correspond to separatrices lying in the topological disk bounded by the cycle $\gamma$).
Denote the sets of coordinates of these points on the circle by $\mathcal{A}^+ \vcentcolon = \{T^+(\mathfrak{c}_r) \; \mbox{(mod 1)} \mid 1 \le r \le k\}$ and $\mathcal{A}^- \vcentcolon = \{T^-(\mathfrak{d}_s) \; \mbox{(mod 1)} \mid 1 \le s \le n\}$.

\defi We call the sets $\mathcal{A}^+$ and $\mathcal{A}^-$ {\itshape non-synchronized}, if for every $\alpha \in \mathbb{R}$ the inequality
$$
\# \left( (\mathcal{A}^+ + \alpha) \cap \mathcal{A}^- \right) \le 1
$$
holds.

\begin{figure}[p]
    \centering
    \begin{subfigure}{\textwidth}
        \centering
        \includegraphics[scale=0.63]{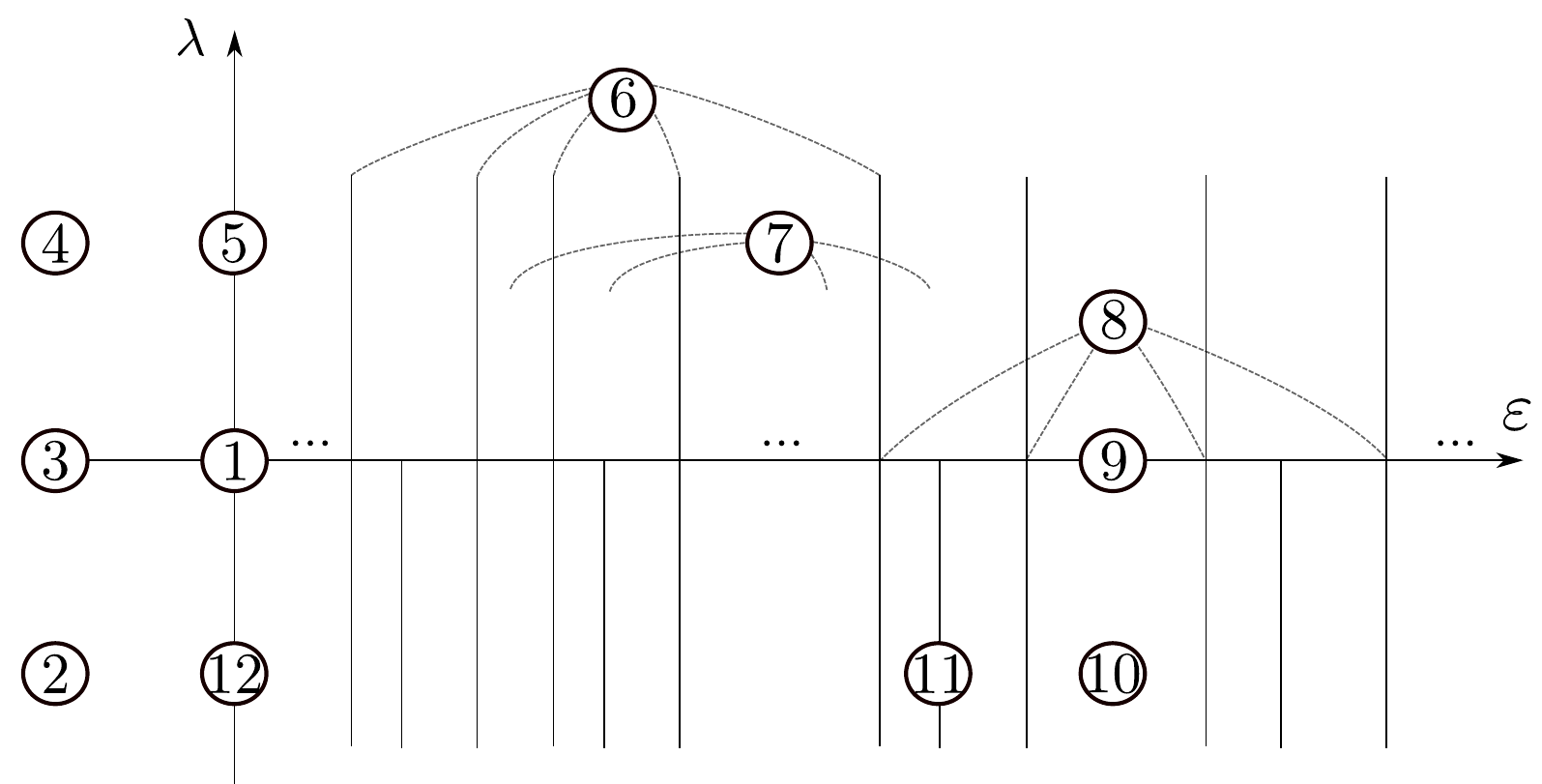}
        \caption{Bifurcation diagram for a PC-HC family of configuration 111. \\ Numbers 1 and 8 correspond to fields with codimension-2 degeneracies. \\ Numbers 3, 5, 6, 9, 11, and 12 correspond to fields with codimension-1 degeneracies. \\ Numbers 2, 4, 7, and 10 correspond to structurally stable fields.}
        \label{Bifurcation diagramm (introduction)}
    \end{subfigure}
    \par\vspace{60pt}
    \begin{subfigure}{\textwidth}
        \centering
        \includegraphics[scale=0.63]{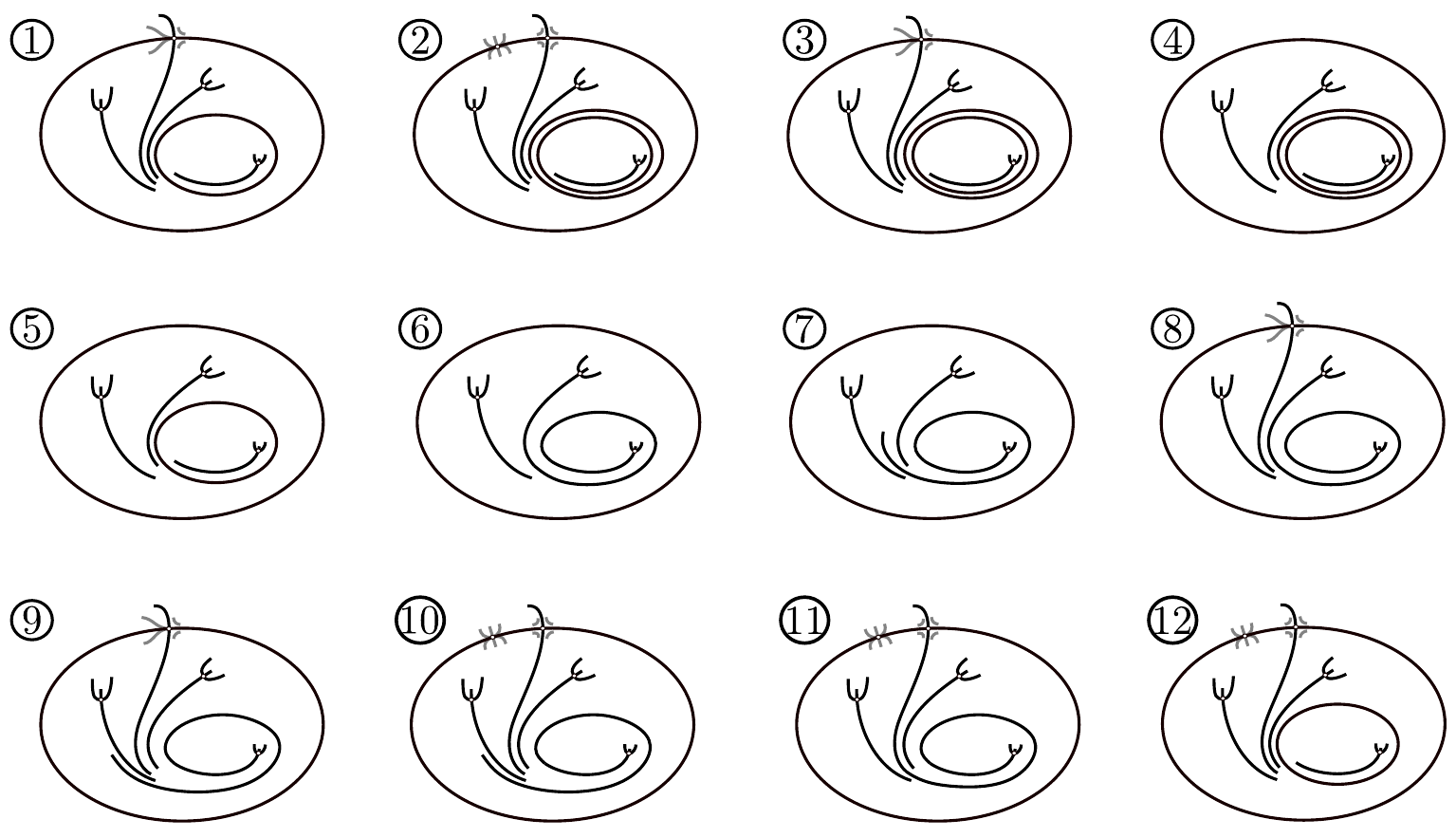}
        \caption{Phase portraits of fields from a PC-HC family of configuration 111.}
        \label{Bifurcation phase portraits (introduction)}
    \end{subfigure}
    \caption{Bifurcation scheme}
\end{figure}

Introduce an equivalence relation on the four obtained sets of coordinates: $\mathcal{L}_1$, $\mathcal{L}_2$, $\mathcal{A}^+$, and $\mathcal{A}^-$. We say that two points from the same set belong to the same equivalence class if the separatrices defining them belong to the same saddle.

Each such equivalence class may contain either one or two points, and pairs of points from different equivalence classes cannot interleave: if there are two equivalence classes $\{a, b\}$ and $\{c, d\}$, then either $(a, b) \subset (c, d)$, or $(c, d) \subset (a, b)$, or $(a, b) \cap (c, d) = \emptyset$.

\defi A finite set on a circle or an interval, partitioned into non-interleaving one- and two-element equivalence classes, will be called {\itshape marked}.

We also introduce an additional {\itshape liaison relation} between the sets $\mathcal{L}_2$ and $\mathcal{A}^+$: we say that points $\mathfrak{b}_j \in \mathcal{L}_2$ and $\mathfrak{c}_r \in \mathcal{A}^+$ are in the liaison relation if the separatrices defining them belong to the same saddle.

Define for the collections $\mathcal{A}^+$ and $\mathcal{A}^-$ the set of pairwise differences of their elements:
$$
\Delta(\mathcal{A}^+, \mathcal{A}^-) := \{T^+(\mathfrak{c}_i) - T^-(\mathfrak{d}_j) \; \mbox{(mod 1)} \mid 1 \le i \le k, 1 \le j \le n \}.
$$

\defi Two marked sets $\mathcal{L}_1$ and $\mathcal{L}_2$ on an interval and two marked finite sets $\mathcal{A}^+$ and $\mathcal{A}^-$ on a circle, together with the liaison relation between the sets $\mathcal{L}_2$ and $\mathcal{A}^+$, are called the {\itshape characteristic sets} of a vector field of the PC-HC class.

\defi Two marked sets $\mathcal{L}$ and $\tilde {\mathcal{L}}$ on an interval are called {\itshape equivalent} if they are mapped to each other by a homeomorphism of the interval preserving the partition into equivalence classes.
A pair of finite marked sets $\mathcal{A}^+$ and $\mathcal{A}^-$ on a circle is {\itshape equivalent} to a pair of finite sets $\mathcal{B}^+$ and $\mathcal{B}^-$ if $|\mathcal{A}^+| = |\mathcal{B}^+|$ and $|\mathcal{A}^-| = |\mathcal{B}^-|$, and there exists $\alpha \in \mathbb{R}$ such that the sets $\Delta(\mathcal{A}^+, \mathcal{A}^-)$ and $\Delta(\mathcal{B}^+, \mathcal{B}^-) + \alpha \; \mbox{(mod 1)}$ are enumerated in the same order. Namely, if
$$
\lambda_{ij} = T^+(\mathfrak{c}_i) - T^-(\mathfrak{d}_j) \; \mbox{(mod 1)}, \quad \Delta(\mathcal{A}^+, \mathcal{A}^-) = \{\lambda_{ij} \mid 1 \le i \le k, 1 \le j \le n \}
$$
and
$$
\tau_{ij} = T^+(\tilde{\mathfrak{c}}_i) - T^-(\tilde{\mathfrak{d}}_j) \; \mbox{(mod 1)}, \quad \Delta(\mathcal{B}^+, \mathcal{B}^-) = \{\tau_{ij} \mid 1 \le i \le k, 1 \le j \le n \},
$$
then for $\alpha \in \mathbb{R}$ and any $i$, $j$, $i'$, $j'$ the implication
$$
\lambda_{ij} > \lambda_{i'j'} \; \Longrightarrow \; \tau_{ij} + \alpha \: \mbox{(mod 1)} > \tau_{i'j'} + \alpha \: \mbox{(mod 1)}
$$
holds.
\par Two collections of characteristic sets $\mathcal{L}_1$, $\mathcal{L}_2$, $\mathcal{A}^+$, $\mathcal{A}^-$ and $\tilde {\mathcal{L}}_1$, $\tilde {\mathcal{L}}_2$, $\tilde {\mathcal{A}}^+$, $\tilde {\mathcal{A}}^-$ are called {\itshape equivalent} if the sets on the intervals are pairwise equivalent ($\mathcal{L}_1$ with $\tilde {\mathcal{L}}_1$ and $\mathcal{L}_2$ with $\tilde {\mathcal{L}}_2$), the pairs of sets on the circles are equivalent, and these equivalences respect the liaison relation between the sets: $\mathcal{L}_2$ with $\mathcal{A}^+$ and $\tilde {\mathcal{L}}_2$ with $\tilde {\mathcal{A}}^+$.

\subsection{Genericity Conditions}
\label{condition for a PC-HC family to be generic}

A generic infinitely smooth vector field $v$ on the two-dimensional sphere $S^2$ with a PC-HC class degeneracy satisfies the following conditions:
\begin{enumerate}
	\item the field has one parabolic cycle $\gamma$ of multiplicity 2, i.e., its Poincaré map $\mathcal{P}$ has the form $x \to x + ax^2 +  \dots$, where $a \ne 0$;
    \item the field has one homoclinic curve $\gamma_h$ of the saddle-node singular point $N$, i.e., a curve that tends to the saddle-node point $N$ as $t \to \pm \infty$, tangentially to the center manifold;
	\item the field has only finitely many singular points and limit cycles, and all singular points and cycles are hyperbolic, with the exception of the saddle-node $N$ and the cycle $\gamma$;
	\item there are no saddle connections between saddles and no separatrix loops;
    \item the characteristic sets on the circle $\mathcal{A}^+$ and $\mathcal{A}^-$ are non-synchronized (the Malta--Palis condition).
\end{enumerate}

Consider a two-parameter family $V$ of vector fields on the sphere $S^2$ that unfolds the given field $v$.
Such a family contains a codimension-2 degeneracy and can therefore be parameterized by a two-dimensional variable, which we denote by $\mu = (\varepsilon, \lambda)$. We impose the following genericity conditions on this family:
\begin{enumerate}
	\item The family $V=\{v_{\varepsilon, \lambda}\}_{(\varepsilon, \lambda)}$ is transverse to the hypersurface consisting of vector fields with a degenerate saddle-node singular point. Consider the center manifold $W^c$ of the extended system:
    \begin{equation*}
	    \left\{
	    \begin{aligned}
            \dot x &= v_{\varepsilon, \lambda}(x), \\
            \dot \varepsilon &= 0, \\
            \dot \lambda &= 0.
        \end{aligned}
        \right.
    \end{equation*}
    If the restriction of the local family $V$ to the manifold $W^c$ has the form
	$$
	\dot x = f(x, \varepsilon, \lambda), \quad \dot \varepsilon = 0, \quad \dot \lambda = 0,
	$$
    where $(x, \varepsilon, \lambda) \in (\mathbb{R}^3, 0)$, $f(0, \varepsilon, 0) = 0$, $\frac{\partial f(0, \varepsilon, 0)}{\partial x} = 0$, and $\frac{\partial^2 f(0, \varepsilon, 0)}{\partial x^2} \ne 0$ for every $\varepsilon \in (\mathbb{R}, 0)$, then the condition is equivalent to requiring that
	$$
	\left. \frac{\partial f(0, \varepsilon, \lambda)}{\partial \lambda} \right|_{\lambda=0} \ne 0.
	$$
	\item If $\mathcal{P}_{\varepsilon, \lambda}$ is the Poincaré map of the parabolic cycle $\gamma$, defined on the transversal $\Gamma_2$ and depending on the parameters $\varepsilon$ and $\lambda$, then in the coordinate $x$ on the transversal $\Gamma_2$, the diffeomorphisms $\mathcal{P}_{\varepsilon, \lambda}$ at fixed $\lambda$ form a saddle-node family:
	$$
	\mathcal{P}_{0, \lambda}(0) = 0, \; \left. \frac{\partial \mathcal{P}_{0, \lambda} (x)}{\partial x} \right|_{x=0} = 1, \; \left. \frac{\partial^2 \mathcal{P}_{0, \lambda} (x)}{\partial x^2} \right|_{x=0} \ne 0, \; \left. \frac{\partial \mathcal{P}_{\varepsilon, \lambda} (0)}{\partial \varepsilon} \right|_{\varepsilon=0} \ne 0.
	$$
\end{enumerate}

A description of the bifurcation scenarios that can occur in families of vector fields satisfying the stated genericity conditions will be given in Section~\ref{chapter about bifurcation diagrams}. Figure~\ref{Bifurcation diagramm (introduction)} presents the bifurcation diagram for a family of configuration 111, and Figure~\ref{Bifurcation phase portraits (introduction)} shows the phase portraits of the corresponding perturbed vector fields.

\begin{Theorem}
In generic two-parameter families, the only unperturbed vector fields that simultaneously possess a non-hyperbolic cycle and a homoclinic trajectory of a saddle-node are those satisfying the genericity conditions listed above.
\end{Theorem}

This statement is natural. Indeed, the violation of each of the genericity conditions is an additional equality-type constraint. Two equality-type constraints are already imposed on the vector field: the existence of a non-hyperbolic cycle and a non-hyperbolic singular point. Vector fields for which these two constraints and some additional one are simultaneously satisfied are subject to three equality-type constraints at once. Such fields cannot appear in generic two-parameter families.

However, a rigorous proof of this fact is not given in the present paper. It relies on an infinite-dimensional version of the transversality theorem, which requires that the degenerate fields form a Banach submanifold---a claim that is far from obvious. In the case of one-parameter families, the analogous statement constitutes an essential part of Sotomayor's paper~\cite{Sotomayor}.

\subsection{Classification of PC-HC Families}
We can now state one of the main results of the paper.

\begin{Theorem}
\label{theorem about conjugation of unpertubed fields implies weak equivalence}
Let $V$ and $W$ be two glocal PC-HC families satisfying the genericity conditions listed in Section~\ref{condition for a PC-HC family to be generic}. Suppose their unperturbed fields $v_0$ and $w_0$ are orbitally topologically equivalent and their characteristic sets are equivalent. Then the families $V$ and $W$ are weakly equivalent.
\end{Theorem}

The definition of weak equivalence is recalled below in Section~\ref{subsection with definition of equivences}.

The proof of this theorem uses the notion of large bifurcation support, whose properties were studied in~\cite{Large bifurcation support}.

From Theorem~\ref{theorem about conjugation of unpertubed fields implies weak equivalence}, taking into account the fact that closeness of PC-HC families implies orbital topological equivalence of their unperturbed vector fields (the proof of this fact will be given in Section~\ref{section about proof of structural stability}), it follows easily that

\begin{Theorem}
\label{theorem about structural stability of PC-HC families}
Generic glocal families of the PC-HC class are structurally stable.
\end{Theorem}

\begin{Theorem}
\label{theorem about equivalence of bifurcation diagrams}
Let two generic PC-HC families have the same configuration and equivalent characteristic sets. Then their bifurcation diagrams are homeomorphic and their bifurcation scenarios coincide.
\end{Theorem}

Thus, in a certain sense, PC-HC families are completely determined by their configurations and characteristic sets.

We now proceed to the proofs of these theorems. We will need the definitions of various types of equivalence of families and the definition of the large bifurcation support.

\section{Basic Definitions}

\subsection{Weak and Moderate Equivalences}
\label{subsection with definition of equivences}

We recall several definitions.

\defi Two vector fields $v$ and $w$ are called {\itshape orbitally topologically equivalent} if there exists a homeomorphism of their phase spaces sending the phase curves of $v$ to the phase curves of $w$ while preserving the time orientation.

\defi We say that a vector field $v \in Vect^{\infty}(S^2)$ satisfies the {\itshape \L ojasiewicz inequality} if for every singular point $a$ there exists a neighborhood in which $\|v(x)\| \ge c \|a -x\|^k$ holds for some natural number $k$ and constant $c>0$.
Denote by $Vect^*(S^2)$ the set of all vector fields with finitely many limit cycles and singular points, in a small neighborhood of each of which the \L ojasiewicz inequality holds.

\defi A smooth $k$-parameter {\itshape family $V$ of vector fields} with base $B$ on the two-dimensional sphere is a smooth vector field on the space $S^2 \times B$, where $B$ is an open ball centered at the origin in $\mathbb{R}^k$, tangent to each fiber $S^2 \times \{\mu\}$. Each individual vector field of the family will be denoted by $v_{\mu}$ and viewed as the restriction of $V$ to the fiber $S^2 \times \{\mu\}$. A {\itshape glocal family} (global in the phase variable and local in the parameter) {\itshape of vector fields} is the germ
of a family of vector fields, regarded as a field on $S^2 \times B$, on the set $S^2 \times \{0\}$. In other words, we introduce an equivalence relation on the space of families of vector fields on $S^2$: two families $v_{\mu}$, $\mu \in B$ and $v_{\delta}$, $\delta \in B'$ are considered equivalent (they define the same germ) if there exists a neighborhood of the origin $U \subset B \cap B'$ such that $v_{\mu} \equiv w_{\mu}$ for all $\mu \in U$; the equivalence class under this relation is what we call a {\itshape glocal family}.

\defi \label{weak equivalence definition} Two \hfill families \hfill of \hfill vector \hfill fields \hfill $V=\{v_{\mu} \mid \mu \in B\}$ \hfill and \\ $W=\{w_{\delta} \mid \delta \in B'\}$ with bases $B$ and $B'$---neighborhoods of the origin in $\mathbb{R}^n$---defined on the two-dimensional sphere, are called {\itshape weakly equivalent} if there exists a mapping
\begin{equation}
    H \colon B \times S^2 \to B' \times S^2, \quad H(\mu, x) = (h(\mu), H_{\mu}(x))
    \label{weak equivalence definition map}
\end{equation}
such that $h$ is a homeomorphism of parameter bases sending 0 to 0, and for each $\mu \in B$ the mapping $H_{\mu} \colon S^2 \to S^2$ provides orbital topological equivalence of the fields $v_{\mu}$ and $w_{h(\mu)}$.

\defi A family $V$ of vector fields on the two-dimensional sphere is {\itshape structurally stable} if there exists a neighborhood $\mathcal{V}$ of the family $V$ in the natural topology such that any family $W$ in the neighborhood $\mathcal{V}$ is weakly equivalent to $V$.

Weak equivalence does not control the continuity of the conjugacy in the parameter at bifurcation points. To work with large bifurcation supports, we need a stronger notion---{\itshape moderate equivalence}, which requires continuity on sets that are essential for bifurcations (the separatrix skeleton and closures of cycles and separatrices).

\defi Two \hfill families \hfill of \hfill vector \hfill fields \hfill $V=\{v_{\mu} \mid \mu \in B\}$ \hfill and \\ $W=\{w_{\delta} \mid \delta \in B'\}$ with bases $B$ and $B'$---neighborhoods of the origin in $\mathbb{R}^n$---defined on the two-dimensional sphere, are called {\itshape moderately equivalent} if the conditions of Definition~\ref{weak equivalence definition} hold, and moreover the mapping~(\ref{weak equivalence definition map}) is continuous in $\mu$ and $x$ on the set
\begin{equation}
    S(v_0) \cup \partial (\overline{PerV} \cup \overline{SepV}) \cap \{\mu = 0\}.
    \label{set from definition of moderate equivalence}
\end{equation}
Furthermore, the inverse of~(\ref{weak equivalence definition map}), the mapping $H^{-1}$, is continuous on the set
\begin{equation}
    S(w_0) \cup \partial (\overline{PerW} \cup \overline{SepW}) \cap \{\delta = 0\},
    \label{the reverse set from definition of moderate equivalence}
\end{equation}
where $S(v)$ denotes the separatrix skeleton, and $PerV$ and $SepV$ denote the unions of all cycles and separatrices of the vector fields of the family $V$, respectively.

\defi \label{moderate equivalence near closed subsets definition} Two families of vector fields $V=\{v_\mu\}$ and $W=\{w_\delta\}$ on the two-dimensional sphere with parameter bases $B$ and $\tilde B$---fixed neighborhoods of the origin in $\mathbb{R}^k$---are called {\itshape moderately equivalent in neighborhoods of closed subsets $Z_1$ and $Z_2$} if
\begin{itemize}
	\item the set $Z_1$ is invariant under the field $v_0$, and $Z_2$ is invariant under the field $w_0$;
	\item there exists a neighborhood $\mathcal{U}$ of the set $Z_1$ and a mapping
	$$
	\mathbf{H} \colon B \times \mathcal{U} \to \tilde B \times S^2, \quad (\mu, x) \mapsto (h(\mu), \mathbf{H}_\mu(x)),
	$$
	which consists of a homeomorphism $h$ with $h(0)=0$, and a collection of mappings $\mathbf{H}_\mu$ such that for each $\mu \in B$, $\mathbf{H}_\mu$ is an orbital topological equivalence of the fields $\left. v_\mu \right|_{\mathcal{U}}$ and $\left. w_{h(\mu)} \right|_{\mathbf{H}_\mu(\mathcal{U})}$;
	\item $\mathbf{H}_0(Z_1)=Z_2$;
	\item the image of any neighborhood $\mathcal{V}$ of the set $\{ \mu=0\} \times Z_1$ under the mapping $\mathbf{H}$ contains some neighborhood of the set $\{\delta = 0\} \times Z_2$; the same holds for $\mathbf{H}^{-1}$;
	\item the mapping $\mathbf{H}$ is continuous on the intersection of its domain with the set~(\ref{set from definition of moderate equivalence}), and $\mathbf{H}^{-1}$ is continuous on the intersection of its domain with the set~(\ref{the reverse set from definition of moderate equivalence}).
\end{itemize}

\subsection{Large Bifurcation Supports: Definition}
\label{subsection definition of large bifurcation support}

\defi (\cite{Large bifurcation support}) The {\itshape extended large bifurcation support ELBS($v_0$)} of a vector field $v_0$ on the sphere is the union of all non-hyperbolic singular points and non-hyperbolic limit cycles of $v_0$, together with the closure of the set of all regular points whose $\alpha$- and $\omega$-limit sets are interesting.

\vspace{-5pt}
In the case of a field $v_0$ of the PC-HC class, the {\itshape interesting sets} are all saddles, the saddle-node $N$, and the parabolic limit cycle $\gamma$, provided that at least one separatrix winds onto it from both the outside and the inside of the cycle. The general definition of interesting sets is given in~\cite{Large bifurcation support}.

\vspace{-5pt}
In the notation introduced below, for a field $v_0$ of the PC-HC class, we introduce several additional notations for elements of the phase portrait of the field $v_0$. Consider all saddles of this field whose separatrices tend to the saddle-node $N$ or the limit cycle $\gamma$. Denote those lying in the disk bounded by the curve $\overline{\gamma}_h$ and not containing the cycle $\gamma$ by $A=\{A_i\}$, and their separatrices tending to the saddle-node $N$ by $\{a_i\}$ (choose the index $i$ to enumerate exactly the separatrices of the saddles: if both unstable separatrices $a_j$ and $a_k$ of some saddle tend to the saddle-node $N$, then this saddle has two different labels $j$ and $k$). Denote the saddles whose separatrices tend to the saddle-node $N$ lying on the opposite side of the curve $\overline{\gamma}_h$ by $B=\{B_j\}$, and their separatrices tending to the singular point $N$ by $\{b_j\}$. Denote the saddles whose separatrices tend to the limit cycle $\gamma$ from the outside and from the inside by $C=\{C_r\}$ and $D=\{D_s\}$ respectively, and their separatrices by $\{c_r\}$ and $\{d_s\}$ in the same manner. Denote by $G$ the union of trajectories of the field $v_0$ that wind off the limit cycle $\gamma$ and tend to the saddle-node $N$ in positive time, with the exception of the separatrix $\beta_1$:
$$G \coloneqq \{ x \in S^2 \mid \alpha(x) = \gamma, \: \omega(x) = N \} \setminus \beta_1.$$

Assume $\# \{a_i\} = l$, $\# \{b_j\} = m$, $\# \{c_r\} = k$, $\# \{d_r\} = n$.

By definition, the set ELBS($v_0$) contains the curves $\overline{\gamma}_h$ and $\gamma$, all trajectories of the set $G$, all separatrices tending to the saddle-node $N$ or the cycle $\gamma$ in forward and backward time, together with their saddles. The separatrix $\beta_1$ of the saddle-node $N$ belongs to ELBS($v_0$) if it winds off the cycle $\gamma$, and does not belong to it otherwise.

In the notation introduced above, the set ELBS($v_0$) takes the form:
\begin{equation}
ELBS(v_0) = \overline{\gamma}_h \cup \gamma \cup \beta_1 \cup G \cup \bigcup \limits_1^l (A_i \cup a_i) \cup \bigcup \limits_1^m (B_j \cup b_j) \cup \bigcup \limits_1^k (C_r \cup c_r) \cup \bigcup \limits_1^n (D_s \cup d_s),
\end{equation}
where the separatrix $\beta_1$ may or may not belong to ELBS($v_0$) depending on the configuration of the field $v_0$.

\defi The {\itshape large bifurcation support} $LBS(V)$ of a local family $V = \{v_\mu\}$ is defined as
\begin{equation}
    LBS(V) = ELBS(v_0) \cap (Sing\: v_0 \cup ( \overline{Per\:V} \cup \overline{Sep \: V}) \cap \{\mu = 0\}),
\end{equation}
where $Sing\: v_0$ denotes the union of all singular points of the unperturbed field $v_0$.

\prop \label{Set G in SepV} For a generic family $V$ of the PC-HC class, the set of trajectories $G$ belongs to the set $\overline{Sep \: V}$.

\vspace{-5pt}
This statement follows from the fact that every trajectory in $G$ is a limit of saddle connections arising under perturbation of the parameter. A formal proof of this proposition will be given in Section~\ref{chapter about bifurcation diagrams}.

Using \hfill this \hfill statement, \hfill we \hfill can \hfill see \hfill that \hfill the \hfill set \\ $(\overline{Per \: V} \cup \overline{Sep \: V}) \cap \{\mu=0\}$ coincides with ELBS($v_0$). Thus, for PC-HC families $V$, the large bifurcation support coincides with the extended large bifurcation support:
$$
LBS(V) = ELBS(v_0).
$$

\subsection{Theorem on Large Bifurcation Supports}

The main property of large bifurcation supports is formulated in the following theorem.

\begin{Theorem}
    \label{main theorem about large bifurcation support} (\cite{Large bifurcation support}) Let vector fields $v_0$ and $w_0$ be orbitally topologically equivalent on the sphere $S^2$. Denote the corresponding homeomorphism by $\hat H$. Let $V$ and $W$ be local families from $Vect^*(S^2)$ that unfold the fields $v_0$ and $w_0$. Suppose there exists a moderate equivalence mapping $\mathbf{H}$ of the families $V$ and $W$ in neighborhoods of their large bifurcation supports in the sense of Definition~\ref{moderate equivalence near closed subsets definition}, and moreover $\left. \hat H \right|_{\mathcal{U}} = \mathbf{H}_0$, where $\mathcal{U}$ is a neighborhood of the large bifurcation support LBS($V$).
\vspace{-5pt}

Then the families $V$ and $W$ are weakly equivalent on the entire sphere $S^2$.
\end{Theorem}

\section{Structural Stability}
\label{section about proof of structural stability}

\begin{Theorem}
    \label{theorem characteristic sets implies moderate equivalence} Let $V$ and $\tilde V$ be two generic PC-HC families with the same configuration and equivalent characteristic sets. Then these families are moderately equivalent in some neighborhoods of their large bifurcation supports.
\end{Theorem}

To prove Theorem~\ref{theorem characteristic sets implies moderate equivalence}, we construct a neighborhood $U$ of the large bifurcation support LBS($V$) and explicitly define the moderate equivalence mapping $\mathbf{H}$. The construction consists of two steps: choosing a suitable neighborhood $U$ covering LBS($V$) in a special way (described in Section~\ref{subsection about choice of neighbourhood of LBS}), and defining the homeomorphism $\mathbf{H}_{\mu}$ on each component of the domain $U$ using normal forms and preservation of relative distance (carried out in Section~\ref{subsection about moderate equivalence in neighbourhoods of LBS}).

\subsection{Neighborhood of the Large Bifurcation Support}
\label{subsection about choice of neighbourhood of LBS}

Denote the desired neighborhood of the large bifurcation support LBS($V$) by $U$. We construct it in parts.

Consider the saddle-node singular point $N$. In some neighborhood $U_N$ of it, the family $V$ is orbitally smoothly equivalent to the family given by the equations:
\begin{equation}
	\left\{
	\begin{aligned}
	\dot {\tilde x} & = \frac{ \tilde x^2 + \lambda}{1 + a(\mu) \tilde x} \\
	\dot {\tilde y} & = - \tilde y,
	\end{aligned}
	\right.
\label{saddle-node normal form}
\end{equation}
where $a(\mu)$ is a smooth function of the two-dimensional parameter $\mu$.

Choose the transversal $\Gamma_1$ close to the saddle-node $N$ in the parabolic sector of the saddle-node, i.e., draw the curve $\Gamma_1$ so that all phase curves of the field $v_0$ tending to the saddle-node $N$ in forward time intersect it, and so that the stable separatrices $\beta_2$ and $\beta_1$ of the saddle-node $N$ intersect it at points $\mathfrak{a}$ and $\mathfrak{b}$ (these points will be the endpoints of the curve $\Gamma_1$). We may assume that the transversal $\Gamma_1$ lies on the boundary $\partial U_N$ of the domain $U_N$ (otherwise, shrink the neighborhood $U_N$). Also choose the transversal $\Gamma_3$ close to the saddle-node $N$ in the hyperbolic sector of the saddle-node, i.e., draw the curve $\Gamma_3$ so that it intersects the unstable separatrix of the saddle-node $N$. Similarly, we may assume that the transversal $\Gamma_3$ lies on the boundary $\partial U_N$ of the neighborhood $U_N$. (See Figure~\ref{Neighbourhood of LBS (conf C)}.)

\begin{figure}[t]
\center{\includegraphics[scale=0.6]{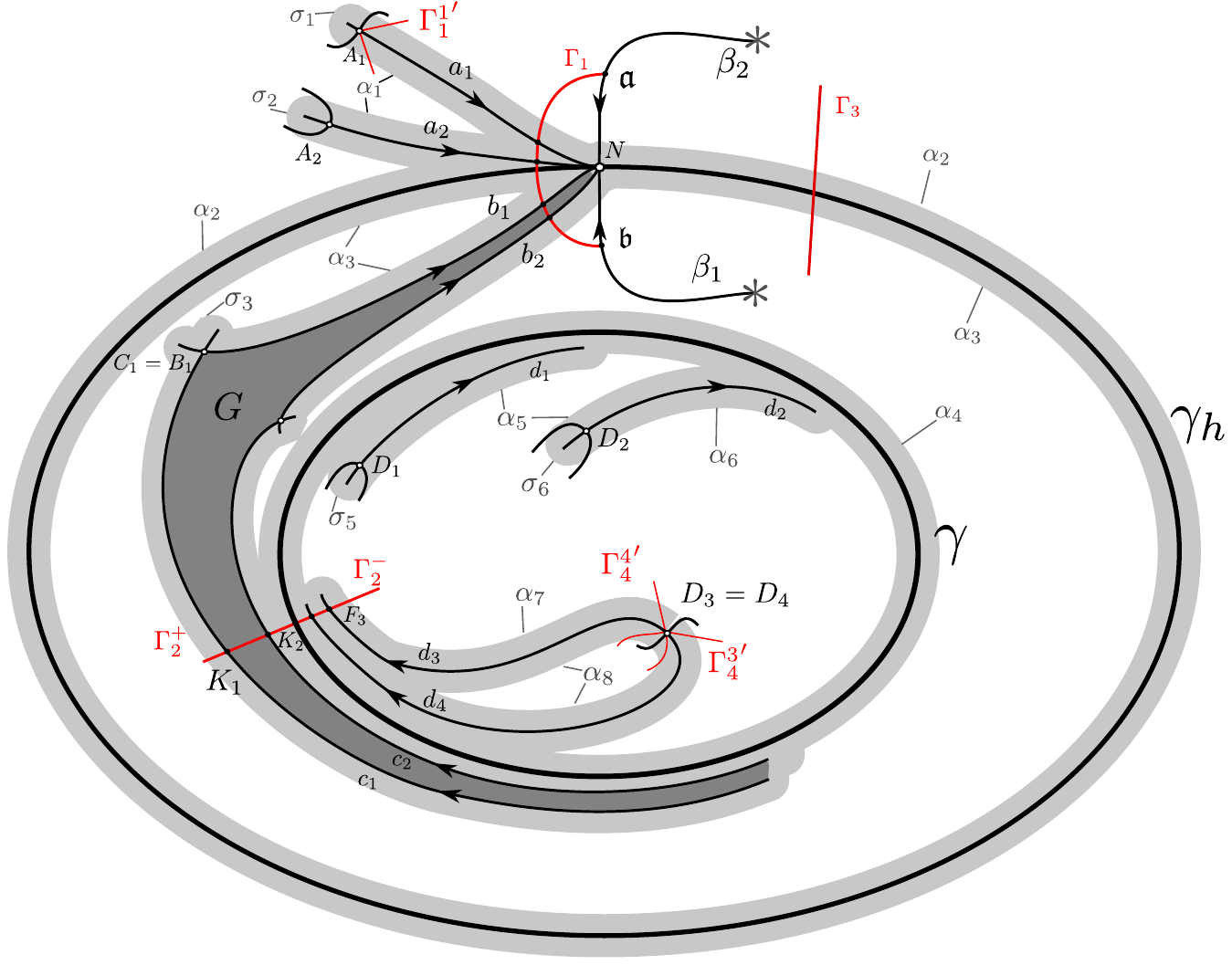}}
\caption{Neighborhood of the large bifurcation support of a family of configuration 101. The light gray region represents the neighborhood $U$. The dark gray region is the set $G$, which belongs to the large bifurcation support.
The bold black lines depict parts of the large bifurcation support, as well as the separatrices whose $\alpha$- or $\omega$-limit sets are interesting.}
\label{Neighbourhood of LBS (conf C)}
\end{figure}

In some neighborhood of the intersection point of the curve $\gamma_h$ with the transversal $\Gamma_3$, the mapping $g_{v_{\mu}} \colon \Gamma_3 \to \Gamma_1$ along the phase curves of the vector field $v_{\mu}$ is defined, i.e., the mapping that sends a point $x$ on the transversal $\Gamma_3$ to the point formed by the intersection of the phase curve passing through $x$ with the transversal $\Gamma_1$. Fix a neighborhood $U_{\gamma_h}$ of the part of the curve $\gamma_h$ bounded by the transversals $\Gamma_3$ and $\Gamma_1$ so that the integral curves of the unperturbed field $v_0$ that intersect $\Gamma_1$ and $\Gamma_3$ enter the region $U_{\gamma_h}$ only through points of the curve $\Gamma_3$ and leave this region through the remaining parts of its boundary. This property is preserved under small perturbation, so it may be assumed to hold for any field of the glocal family $V$.

Consider the separatrix $a_i$ tending to the saddle-node $N$ in positive time. Choose its neighborhood $U_{a_i}$ so that it is bounded near the saddle-node $N$ by the transversal $\Gamma_1$, and near the saddle $A_i$ by the stable manifold of that saddle. The two other boundary components of $\partial U_{a_i}$ are smooth curves $\alpha_i$ and $\alpha_{i+1}$, close to the separatrix $a_i$; choose these curves so that the phase curves of the unperturbed field $v_0$ intersecting them enter the domain $U_{a_i}$ through them. Construct the neighborhoods $U_{a_j}$ for all other separatrices $\{a_j\}$ analogously. The curve $\alpha_l$ (where $l = \# \{a_i\}$), lying on the boundary $\partial U_{a_l}$ of the neighborhood of the separatrix $a_l$, terminates at the transversal $\Gamma_1$ and intersects the boundary of the neighborhood $U_{\gamma_h}$. We assume that the curve lying on the boundary of the domain $U_{\gamma_h}$ is a continuation of the curve $\alpha_l$. Then the curve $\alpha_l$ starts at the stable manifold of the saddle $A_l$, passes near the separatrix $a_l$, goes around the homoclinic curve $\gamma_h$, and terminates at the transversal $\Gamma_3$, with all phase curves intersecting $\alpha_l$ entering the region it bounds. (In Figure~\ref{Neighbourhood of LBS (conf C)}, the curve $\alpha_2$ has this property.)

For saddles $\{A_i\}$ for which only one separatrix tends to the saddle-node $N$, take the neighborhood $U_{A_i}$ to be a ball in $\mathbb{R}^2$ of sufficiently small radius. It is divided by the stable manifold of the saddle $A_i$ into two half-disks; one of them belongs to the neighborhood $U_{a_i}$, and the other is bounded by the curve $\sigma_i$ (denote this half-disk by $\Sigma_{A_i}$). The curves $\sigma_i$ are tangent to each field $v_{\mu}$ of the family $V$ at exactly two points.

For the separatrices $\{d_i\}$ and saddles $\{D_i\}$ lying inside the topological disk bounded by the parabolic cycle $\gamma$, construct the neighborhoods $U_{d_i}$ and $U_{D_i}$ analogously. The boundaries of the sets $U_{d_i}$ are the curves $\alpha_{l+m+k+1}$, \dots, $\alpha_{l+m+k+n}$, as well as $\sigma_{l+m+k+1}$, \dots, $\sigma_{l+m+k+n}$ (the curves $\sigma_i$ appear as boundaries of those saddles for which only one separatrix belongs to the large bifurcation support).

For the separatrices $\{b_i\}$ and $\{c_i\}$ and their saddles, the neighborhoods are chosen analogously. If the separatrices $b_i$ and $c_j$ share a common saddle and lie on the boundary of a canonical domain belonging to the set $G$, then we choose a neighborhood of the entire canonical domain, i.e., the sets $U_{b_i}$ and $U_{c_j}$ are defined so that their boundary includes on one side the two curves $\alpha_{l+i}$ and $\alpha_{l+m+j}$, analogous to those constructed above, a segment of the transversal $\Gamma_1$, and a segment of the transversal $C^1$ (defined below), and on the other side the pair of separatrices $b_{i+1}$ and $c_{j+1}$, which lie on the boundary of the canonical domain belonging to the set $G$. Thus, in Figure~\ref{Neighbourhood of LBS (conf C)}, the neighborhoods $U_{b_1}$ and $U_{c_1}$ of the separatrices $b_1$ and $c_1$ are bounded by the curves $\alpha_3$ and $\alpha_4$ and cover the set $G$ up to the part of its boundary consisting of the separatrices $b_2$, $c_2$, and their common saddle $B_2 = C_2$.
If the family under consideration belongs to configurations with $a_2 = 0$, then the curve $\alpha_{l+1}$, starting at the stable manifold of the saddle $B_1$, intersects the transversal $\Gamma_1$ and the curve bounding the neighborhood $U_{\gamma_h}$ of the homoclinic trajectory $\gamma_h$. We assume that the curve bounding the domain $U_{\gamma_h}$ is a continuation of the curve $\alpha_{l+1}$. Then the curve $\alpha_{l+1}$ starts at the stable manifold of the saddle $B_1$, passes along the separatrix $b_1$, then along the homoclinic trajectory $\gamma_h$, then along the separatrix $b_m$, and terminates at the stable manifold of the saddle $B_m$. In Figure~\ref{Neighbourhood of LBS (conf C)}, the curve $\alpha_3$ has this property.

Let us draw two transversal loops $C^1$ and $C^2$ in a neighborhood of the parabolic cycle $\gamma$. Denote the annulus bounded by them by $U_\gamma$. Choose the loops $C^1$ and $C^2$ so that all phase curves of the field $v_0$ enter the domain $U_\gamma$ only through one of the curves and leave it only through the other.
As before, if the family belongs to configurations with $a_2 = 0$, then the curve $\alpha_{l+m-1}$, which is the boundary of the neighborhood $U_{c_1}$ of the separatrix $c_1$, intersects the curve $C^1$. We assume that a segment of the curve $C^1$ is a continuation of the curve $\alpha_{l+m-1}$. Then the curve $\alpha_{l+m-1}$ starts at the unstable manifold (in the case of configurations with $a_3 = 1$) or the stable manifold (in the case of $a_3 = 0$) of the saddle $C_1$, passes along the separatrix $c_1$, then along the parabolic cycle $\gamma$, then along the separatrix $c_m$, and terminates at the unstable manifold (for $a_3 = 1$) or the stable manifold (for $a_3 = 0$) of the saddle $C_m$. In Figure~\ref{Neighbourhood of LBS (conf C)}, the curve $\alpha_4$ plays this role.

Define the neighborhoods $U_{d_i}$ of the separatrices $d_i$ analogously. Their boundaries are segments of the curves $\alpha_{l+m+i-1}$ and $\alpha_{l+m+i}$. We assume that the curve $\alpha_{l+m+i}$ starts at the stable manifold of the saddle $D_i$, goes along the separatrix $d_i$ up to the curve $C^2$, then passes along a segment of the curve $C^2$ bounded by the points $\{d_i\} \cap C^2$ and $\{d_{i+1}\} \cap C^2$, then along the separatrix $d_{i+1}$, and terminates at the stable manifold of the saddle $D_{i+1}$. In Figure~\ref{Neighbourhood of LBS (conf C)}, the curves $\alpha_5$, $\alpha_6$, and $\alpha_7$ have this property.

Thus, we have constructed the neighborhood $U$ of the large bifurcation support of the family $V$, which can be expressed as the union of the sets $U_{A_i}$, $U_{a_i}$, $U_{B_i}$, $U_{b_i}$, $U_{C_i}$, $U_{c_i}$, $U_{D_i}$, $U_{d_i}$, $U_{N}$, $U_{\gamma_h}$, $U_{\gamma}$; and its boundary consists of the curves $\alpha_i$ and arcs $\sigma_i$, which are transverse to every field of the family $V$. Let $\tilde U$ be the analogously constructed neighborhood for the family $\tilde V$.

{\itshape Remark:} the stable separatrix $\beta_1$ of the saddle-node $N$ does not belong to the large bifurcation support if the family $V$ has configurations with $a_2 = 0$, and hence requires no separate neighborhood construction. If the family $V$ has configuration with $a_2 = 1$, then the separatrix $\beta_1$ belongs to the large bifurcation support, but since it bounds canonical domains belonging to the set $G$, it is already covered by the neighborhood of the set $G$ and requires no separate construction.

\subsection{Moderate Equivalence Mapping in a Neighborhood of the LBS}
\label{subsection about moderate equivalence in neighbourhoods of LBS}

Let $A_i(\varepsilon, \lambda)$ be the continuation of the saddle $A_i$ of the vector field $v_0$ with respect to the parameter, i.e., the saddle of the vector field $v_{\varepsilon, \lambda}$ such that the point $A_i(\varepsilon, \lambda)$ depends smoothly on the parameters and coincides with $A_i$ at $(\varepsilon, \lambda) = (0, 0)$. Similarly denote the continuations with respect to the parameter of the other elements of the large bifurcation support. Let, for $\lambda < 0$, the points $S(\varepsilon, \lambda)$ and $N(\varepsilon, \lambda)$ be the saddle and the topological node into which the saddle-node $N$ splits under the bifurcation. We have $S(\varepsilon, \lambda) \to N$ and $N(\varepsilon, \lambda) \to N$ as $\lambda \to 0+$. Assume that at $\lambda = 0$ we have $S(\varepsilon, 0) = N(\varepsilon, 0) = N$. For positive values of the parameter $\lambda$, the points $S(\varepsilon, \lambda)$ and $N(\varepsilon, \lambda)$ are not defined.

Denote by $K_i(\varepsilon, \lambda)$ the last intersection of the separatrix $c_i(\varepsilon, \lambda)$ with the transversal $\Gamma_2$, and by $F_j(\varepsilon, \lambda)$ the first intersection of the separatrix $d_j(\varepsilon, \lambda)$ with the transversal $\Gamma_2$. Let the point $K_{k+1}(\varepsilon, \lambda)$ be the last intersection of the separatrix $\beta_1$ with the transversal $\Gamma_2$; it is defined only for non-positive values of the parameter $\lambda$.

Let $\{g_{\varepsilon, \lambda}\}$ and $\{\tilde g_{\varepsilon, \lambda}\}$ be two finite-parameter families of mappings of the form
$$
y \mapsto y + (y^2 + \varepsilon)(1 + f(y, \varepsilon, \lambda)).
$$

Let $K = \{K_1, \ldots, K_k\}$ be a set of points on the interval $[K_1, g_{0, 0}(K_1)] \times \{0, 0\}$ with coordinates $x(K_i)$ satisfying $0 < x(K_j) < x(K_{j+1})$ for any $j \le n-1$. Let $\Gamma_K$ be a collection of smooth hypersurfaces $\Gamma_{K_1}, \ldots, \Gamma_{K_k}$ in the ambient space $(x, \varepsilon, \lambda)$, transverse to the line $(\varepsilon, \lambda)=(0, 0)$, with $K_j \in \Gamma_{K_j}$. Let $\Gamma_F$ be the analogous collection of hypersurfaces with $F_j \in \Gamma_{F_j}$. Let the sets $K$ and $F$ satisfy the non-synchronization condition. Such a family $\{g_{\varepsilon, \lambda}\}$ together with fixed sets $\Gamma_K$ and $\Gamma_F$ is called a {\itshape saddle-node family}. Let $\{\tilde g_{\varepsilon, \lambda}\}$ be another saddle-node family with sets $\tilde K$, $\tilde F$, $\Gamma_{\tilde K}$, and $\Gamma_{\tilde F}$, possessing the same properties.

\thm \label{theorem about saddle-node family of maps} \cite{SN-SN family} \textbf{(Lemma on marked saddle-node families)} There exists a weak equivalence mapping $\mathcal{H}$ of the families $\{g_{\varepsilon, \lambda}\}$ and $\{\tilde g_{\varepsilon, \lambda}\}$ that sends $K_j$ to $\tilde K_j$, $(\Gamma_{K_j}, K_j)$ to $(\Gamma_{\tilde K_j}, \tilde K_j)$, $F_j$ to $\tilde F_j$, and $(\Gamma_{F_j}, F_j)$ to $(\Gamma_{\tilde F_j}, \tilde F_j)$. Moreover, the mapping $\mathcal{H}$ is continuous on the sets $K$, $F$, and at the point $(0, 0, 0)$.

We can now proceed to constructing the moderate equivalence mapping $\mathbf{H}$ in the neighborhood of the large bifurcation supports of the families $V$ and $\tilde V$. To construct it near the parabolic cycle $\gamma$, we will use Theorem~\ref{theorem about saddle-node family of maps} stated above.

Consider the saddle-node family of mappings $\{g_{\varepsilon, \lambda}\}$ given by the action of the Poincaré map on the transversal $\Gamma_2$ in the neighborhood of the limit cycle $\gamma$. As before, let $x$ be the normalizing coordinate on the transversal $\Gamma_2$. We fixed on the transversal $\Gamma_2$ the points $K_1(\varepsilon, \lambda), \ldots, K_{k+1}(\varepsilon, \lambda)$, defined as the points of the last intersection of the separatrices $c_i$ and the separatrix $\beta_1$ with the transversal $\Gamma_2$, and introduced the points $F_1(\varepsilon, \lambda), \ldots, F_k(\varepsilon, \lambda)$ as the points of the first intersection of the separatrices $d_i$ with the transversal $\Gamma_2$. Then, in the notation of Theorem~\ref{theorem about saddle-node family of maps}, we consider the hypersurfaces $\Gamma_K$ and $\Gamma_F$, defined as
$$
\Gamma_{K_j} = \underset{\varepsilon, \lambda} \cup \: K_j(\varepsilon, \lambda), \quad \Gamma_{F_j} = \underset{\varepsilon, \lambda} \cup \: F_j(\varepsilon, \lambda).
$$
In this case, these are two-dimensional objects. The surface $\Gamma_{K_{k+1}}$ is defined only for non-positive values of the parameter $\lambda$. Let $\tilde K$, $\tilde F$, $\Gamma_{\tilde K}$, and $\Gamma_{\tilde F}$ be the analogously defined objects for the family of mappings $\{\tilde g_{\varepsilon, \lambda}\}$ induced by the Poincaré maps of the family $\tilde V$ near the parabolic cycle $\tilde \gamma$.

By Theorem~\ref{theorem about saddle-node family of maps}, we have a weak equivalence mapping $\mathcal{H}$ of the families $\{g_{\varepsilon, \lambda}\}$ and $\{\tilde g_{\varepsilon, \lambda}\}$, which is continuous at the points $K_j(0, \lambda_0)$ and $F_j(0, \lambda_0)$. Accordingly, the inverse mapping $\mathcal{H}^{-1}$ is continuous at the points $\tilde K_j(0, \lambda_0)$ and $\tilde F_j(0, \lambda_0)$. Define the mapping on the parameter base as
$$
h(\varepsilon, \lambda) = (h_\mathcal{H}(\varepsilon), \lambda),
$$
where $h_\mathcal{H}$ is the parameter transformation obtained from the weak equivalence mapping $\mathcal{H}$ of saddle-node families. We use this parameter change $h$ to define the moderate equivalence mapping $\mathbf{H} = (h, \mathbf{H}_{\varepsilon, \lambda})$ of the families $V$ and $\tilde V$ in neighborhoods of their large bifurcation supports. Such a reparameterization ensures that vector fields $v_{\varepsilon, \lambda}$ possessing a saddle connection are mapped to vector fields $\tilde v_{h(\varepsilon, \lambda)}$ also possessing a saddle connection.

Recall that $\Sigma_{A_i}$ are connected neighborhoods of the saddles $A_i$ for which only one separatrix belongs to the large bifurcation support; these domains are bounded by the stable manifolds of the saddles $A_i$, are homeomorphic to a half-disk, and cover the germ of the unstable separatrix not belonging to the large bifurcation support. The sets $\Sigma_{B_i}$, $\Sigma_{C_i}$, and $\Sigma_{D_i}$ are defined analogously.

Here we begin the construction of the homeomorphism $\mathbf{H}_{\varepsilon, \lambda}$. The homeomorphism $\left. \mathbf{H}_{\varepsilon, \lambda} \right| _ {\Sigma_{A_i}} \colon \Sigma_{A_i} \to \tilde \Sigma_{A_i}$, continuous in the parameters and conjugating the vector fields $v_{\varepsilon, \lambda}$ and $\tilde v_{h_{\mathcal{H}}(\varepsilon), \lambda}$, is constructed as the restriction of the homeomorphism conjugating the vector fields in a neighborhood of the hyperbolic point $A_i$ to the domain $\Sigma_{A_i}$. The same applies to the domains $\Sigma_{B_i}$, $\Sigma_{C_i}$, and $\Sigma_{D_i}$.

Let ${\Gamma_1^i}'$ and ${\Gamma_1^i}''$ be transversals to the field $v_0$ connecting the saddle $A_i$ to a point on the boundary of the neighborhood $U_{a_i}$, not intersecting the separatrices of the saddle $A_i$. Similarly, let the transversals ${\Gamma_2^i}'$ and ${\Gamma_2^i}''$ connect the saddles $B_i$ to points on the boundary of the neighborhood $U_{b_i}$, the transversals ${\Gamma_3^i}'$ and ${\Gamma_3^i}''$ connect the saddles $C_i$ to points on the boundary of the neighborhood $U_{c_i}$, and the transversals ${\Gamma_4^i}'$ and ${\Gamma_4^i}''$ connect the saddles $D_i$ to points on the boundary of the neighborhood $U_{d_i}$. (Figure~\ref{Neighbourhood of LBS (conf C)} shows the transversals ${\Gamma_1^1}'$ and ${\Gamma_1^1}''$ near the saddle $A_1$, as well as the transversals ${\Gamma_4^3}'$, ${\Gamma_4^3}''$, ${\Gamma_4^4}'$, and ${\Gamma_4^4}''$ near the coincident saddles $D_3$ and $D_4$.)

In the sector ${S_i}'$ of the neighborhood of the saddle $A_i$, bounded by the separatrix of the saddle, the transversal ${\Gamma_1^i}'$, and the boundary of the set $U_{a_i}$, construct the mapping $\left. \mathbf{H}_{\varepsilon, \lambda} \right| _ {{S_i}'} \colon {S_i}' \to \tilde {S_i}'$ as follows. Every trajectory of the field $v_{\varepsilon, \lambda}$ enters the sector ${S_i}'$ through the arc $\alpha_{i-1}$ on the boundary of the set $U_{a_i}$ and exits through the transversal ${\Gamma_1^i}'$. If the trajectory of $v_{\varepsilon, \lambda}$ intersects $\alpha_{i-1}$ at the point $p(\varepsilon, \lambda)$ and the transversal ${\Gamma_1^i}'$ at the point $q(\varepsilon, \lambda)$, set
\begin{equation}
 \mathbf{H}_{\varepsilon, \lambda}(p(\varepsilon, \lambda)) = \tilde p(h_{\mathcal{H}}(\varepsilon), \lambda), \quad \mathbf{H}_{\varepsilon, \lambda}(q(\varepsilon, \lambda)) = \tilde q(h_{\mathcal{H}}(\varepsilon), \lambda),
 \label{formula for homeomorphism of vector families}
\end{equation}
and define the homeomorphism $\mathbf{H}_{\varepsilon, \lambda}$ on the remaining part of the trajectory $pq$ so that it {\itshape preserves relative distance}, i.e., if the point $r$ divides the arc $pq$ in the ratio $u/(1-u)$, then its image $H(r)$ divides the arc $H(p)H(q)$ in the same ratio $u/(1-u)$.

Now consider the sector ${S_i}''$ of the neighborhood of the saddle $A_i$, bounded by the separatrix of the saddle, the transversal ${\Gamma_1^i}''$, and the boundary of the set $U_{a_i}$. Construct the mapping $\left. \mathbf{H}_{\varepsilon, \lambda} \right| _ {{S_i}''}$ on it in an analogous manner, using formula~(\ref{formula for homeomorphism of vector families}) for the points $p$ and $q$ on the boundary of the sector ${S_i}''$ and defining it on the remaining points of the trajectories $pq$ with the requirement of preserving relative distance.

Analogously, consider the part of the domain $U_{a_i}$ bounded by the transversal ${\Gamma_1^i}'$, the curve $\alpha_{i-1}$, the transversal $\Gamma_1$ near the saddle-node $N$, and the separatrix $a_i$. Every trajectory of the field $v_{\varepsilon, \lambda}$ enters this domain through the curves ${\Gamma_1^i}'$ and $\alpha_{i-1}$, and exits through the transversal $\Gamma_1$. Thus, the homeomorphism $\mathbf{H}_{\varepsilon, \lambda}$ can again be defined using formula~(\ref{formula for homeomorphism of vector families}) for points $p \in {\Gamma_1^i}' \cup \alpha_{i-1}$ and points $q \in \Gamma_1$, and extended to the remaining part by the principle of preserving relative distance. The mapping $\mathbf{H}_{\varepsilon, \lambda}$ on the remaining part of the domain $U_{a_i}$, bounded by the transversal ${\Gamma_1^i}''$, the curve $\alpha_{i}$, the transversal $\Gamma_1$, and the separatrix $a_i$, is defined in exactly the same way.

On the separatrix $a_i(\varepsilon, \lambda)$ itself of the saddle $A_i(\varepsilon, \lambda)$, we also define the homeomorphism $\mathbf{H}_{\varepsilon, \lambda}$ as the mapping that sends separatrices to separatrices while preserving relative distance. This completes the construction of the mapping $\mathbf{H}$ on the set $\cup_i(U_{A_i} \cup U_{a_i})$, i.e., in the neighborhoods of the separatrices $a_i$ up to the transversal $\Gamma_1$.

The mapping $\mathbf{H}$ in the neighborhoods $U_{B_i}$ of the saddles $B_i$ and the neighborhoods $U_{b_i}$ of the separatrices $b_i$ up to the transversal $\Gamma_1$ is constructed by the same principle.

In the neighborhood $U_N$ of the saddle-node $N$, there is a smooth equivalence to the local family given by equation~(\ref{saddle-node normal form}). This defines the mapping $\mathbf{H}$ on the domain $U_N$.

In the neighborhood $U_{\gamma_{h}}$ of the homoclinic curve $\gamma_h$, the following property holds: all trajectories of the field $v_{\varepsilon, \lambda}$ enter the domain $U_{\gamma_{h}}$ only through the transversal $\Gamma_3$ and the curves $\alpha_m$ and $\alpha_{m+1}$, and leave it through the transversal $\Gamma_1$. Therefore, we can define the mapping $\mathbf{H}_{\varepsilon, \lambda}$ using formula~(\ref{formula for homeomorphism of vector families}) for points $p$ on the curves $\Gamma_3$, $\alpha_m$, and $\alpha_{m+1}$ and points $q$ on the curve $\Gamma_1$, and extend it to the remaining points along trajectories by the principle of preserving relative distance.

Now construct the homeomorphism $\mathbf{H}$ in the neighborhood $U_{\gamma}$ of the parabolic cycle $\gamma$. Its restriction to the transversal $\Gamma_2$ has already been obtained from Theorem~\ref{theorem about saddle-node family of maps}. If the family $V$ has configuration with $a_3 = 1$, then the parabolic cycle $\gamma$ is repelling from the side of the saddle-node. Therefore, the backward semi-orbit $\{g_v^t(p)\}_{t \le 0}$ of any point $p \in \Gamma_2^+$ on the semi-transversal $\Gamma_2^+$ winds onto the parabolic cycle $\gamma$ in backward time. Let $q$ be the penultimate intersection of the backward semi-orbit of the point $p$ with the transversal $\Gamma_2$ (the last intersection being the point $p$ itself). By the construction of the mapping $\mathbf{H}$ on the transversal $\Gamma_2$, the equality from formula~(\ref{formula for homeomorphism of vector families}) already holds for the points $p$ and $q$. Define the homeomorphism $\mathbf{H}$ on the trajectory arcs bounded by these points with preservation of relative distance. Define $\mathbf{H}$ analogously for points $p \in \Gamma_2^-$ on the semi-transversal $\Gamma_2^-$ inside the cycle $\gamma$, but here consider not the backward semi-orbits of the points $p$ but the forward ones. On the parabolic cycle itself, define $\mathbf{H}_{0, \lambda}$ also with preservation of relative distance. The mapping constructed in this way is continuous for each fixed value of the parameter $\mu = (\varepsilon, \lambda)$ in a neighborhood of the parabolic cycle by the theorem on continuous dependence of solutions of a differential equation on initial conditions.
Continuity on the separatrix skeleton follows from Theorem~\ref{theorem about saddle-node family of maps}.

Thus, the mapping in the neighborhood $U_\gamma$ of the parabolic cycle $\gamma$ is defined and continuous on the large bifurcation support. (If the family $V$ has configuration with $a_3 = 0$, i.e., its parabolic cycle attracts from the side of the saddle-node, then the same argument applies with the only modification of replacing all considered backward semi-orbits by forward ones, and forward ones by backward ones.)

In the neighborhoods $U_{C_i}$ and $U_{c_i}$ (respectively $U_{D_i}$ and $U_{d_i}$) of the saddles $C_i$ and separatrices $c_i$ (respectively saddles $D_i$ and separatrices $d_i$), the mapping $\mathbf{H}$ is defined in the same manner as for the saddles $A_i$ and separatrices $a_i$.

Note that since we chose the neighborhoods $U_{b_i}$ and $U_{c_i}$ of the separatrices $b_i$ and $c_i$ so that they also cover part of the set $G$, the homeomorphism $\mathbf{H}$ defined on the domains $U_{b_i}$ and $U_{c_i}$ is also defined on part of the set $G$. The remaining part of the set $G$ is covered by the sets $U_N$ and $U_{\gamma}$, on which the homeomorphism $\mathbf{H}$ is also defined. Therefore, the mapping $\mathbf{H}$ is defined throughout the entire neighborhood $U$ of the large bifurcation support of the family $V$.

\textbf{Proof of Theorem~\ref{theorem characteristic sets implies moderate equivalence}.} Let $V$ and $\tilde V$ be generic PC-HC families with equivalent characteristic sets and the same configuration. Take the neighborhood $U$ of the large bifurcation support LBS($V$) constructed in Section~\ref{subsection about choice of neighbourhood of LBS}, and define on it the homeomorphism $\mathbf{H}$ as described above. From Theorem~\ref{theorem about saddle-node family of maps} and the explicit construction of the homeomorphism $\mathbf{H}$, we can see that this homeomorphism is continuous at the points of the set~(\ref{set from definition of moderate equivalence}). As the neighborhood $W$ of the large bifurcation support LBS($\tilde V$) of the family $\tilde V$, take the image of the neighborhood $U$ under the mapping $\mathbf{H}$. It is readily verified that the inverse mapping $\mathbf{H}^{-1} \colon W \to U$ is also continuous at the points of the set~(\ref{the reverse set from definition of moderate equivalence}). \qed

\textbf{Proof of Theorem~\ref{theorem about conjugation of unpertubed fields implies weak equivalence}.} Let the unperturbed fields $v_0$ and $w_0$ of two generic PC-HC families $V$ and $W$ be close and orbitally topologically equivalent. Then their configurations coincide and their characteristic sets are equivalent. Hence, by Theorem~\ref{theorem characteristic sets implies moderate equivalence}, there exists a moderate equivalence homeomorphism $\mathbf{H}$ in neighborhoods of the large bifurcation supports of the families $V$ and $W$. But then, by Theorem~\ref{main theorem about large bifurcation support}, the families $V$ and $W$ are weakly equivalent on the entire sphere $S^2$. \qed

\subsection{Conditional Structural Stability of Generic Vector Fields of the PC-HC Class}
\label{subsection about final proof of structural stability of PC-HC families}

Consider the set of all vector fields on the sphere with a generic PC-HC degeneracy. It forms a subspace $\mathfrak{S}_{PC-HC}$ of codimension 2 in the space $Vect^*(S^2)$. In this section we show that generic PC-HC vector fields are conditionally structurally stable, i.e., structurally stable under perturbations within the space $\mathfrak{S}_{PC-HC}$. Using the classical terminology of A.\,A.~Andronov, we can say that we prove stability in the space of fields of second degree of instability (non-roughness).

\lemma \label{lemma about structural stability of vector fields} Let $V$ and $\tilde V$ be two PC-HC families close in the natural topology, and let $v_0$ and $\tilde v_0$ be their unperturbed vector fields. Then $v_0$ and $\tilde v_0$ are orbitally topologically equivalent.

{\itshape Proof.} Divide the phase space $S^2$ into two subsets $U_1$ and $U_2$. Let $U_1$ be a neighborhood of the set $\overline{\gamma}_h \cup \gamma$, consisting of the union of the neighborhoods of these two curves chosen above: $U_1 = U_N \cup U_{\gamma_h} \cup U_\gamma$.
And let the set $U_2$ be the complement of $U_1$.

By definition of a field $v_0$ of the PC-HC class, all singular points of the field $\left. v_0 \right|_{U_2}$ are hyperbolic, all closed orbits of the field are limit cycles, and there are no saddle connections. Under small perturbation, the $\alpha$- and $\omega$-limit sets of any separatrix are preserved. It follows that the field $\left. v_0 \right|_{U_2}$ is structurally stable. In particular, there exists a mapping $\left. \hat H \right|_{U_2}$ of orbital equivalence between the fields $\left. v_0 \right|_{U_2}$ and $\left. \tilde v_0 \right|_{\hat H(U_2)}$.

The mapping $\hat H$ is defined on the set $U_2$, and hence can be extended continuously to the boundary of the set $U_1$. Define the mapping $\hat H$ in the neighborhood $U_\gamma$ of the parabolic cycle $\gamma$. Consider the positive semi-orbit $\{g_t(p)\}_{t \ge 0}$ of an arbitrary point $p$ on the boundary of the domain $U_\gamma$. Let this trajectory intersect the transversal $\Gamma_2$ at the points $\{p_n\}_{n \in \mathbb{N}}$. This sequence converges to the limit point $q = \Gamma_2 \cap \gamma$. Let $\tilde p$ be the image of $p$ under $\hat H$, and let $\{\tilde p_n\}_{n \in \mathbb{N}}$ correspond to the intersection points of the trajectory $\{g_t(\tilde p)\}_{t \ge 0}$ with the transversal $\tilde \Gamma_2$. Define the mapping $\hat H$ so that it sends the arcs $pp_1$ and $p_ip_{i+1}$ of the trajectory to the arcs $\tilde p \tilde p_1$ and $\tilde p_i \tilde p_{i+1}$ of the trajectory with preservation of relative distance. Define the mapping $\hat H$ in this way for all points $p$ on the boundary of the domain $U_\gamma$, considering forward or backward semi-orbits of the points depending on which side the limit cycle $\gamma$ is repelling from. By the theorem on continuous dependence of solutions on initial conditions, the resulting mapping is continuous on $U_\gamma \setminus \gamma$. Extend the mapping to the cycle $\gamma$ itself also by continuity.

In the neighborhood of the curve $\overline{\gamma}_h$, extend the mapping by sending the trajectory passing through a point $p$ on the boundary of the domain $U_{\gamma} \cup U_N$ to the trajectory passing through the point $\hat H (p)$ on the boundary of the domain $\hat H (U_{\gamma} \cup U_N)$, with preservation of relative distance. In this process, the segments of the separatrices $\beta_1$ and $\beta_2$ are mapped to the segments of the separatrices $\tilde \beta_1$ and $\tilde \beta_2$. Thus, the mapping $\hat H$ turns out to be defined on all of $S^2$, except for the curve $\overline{\gamma}_h$, to which the mapping can be extended by continuity. \qed

Thus, using Theorem~\ref{main theorem about large bifurcation support}, we can conclude that a generic family $V$ of the PC-HC class is structurally stable. This completes the proof of Theorem~\ref{theorem about structural stability of PC-HC families}.

\section{Classification of Large Bifurcation Supports of PC-HC Families}

In this section we show that the combinatorial data---the configuration and the equivalence class of the characteristic sets---are not only invariants of moderate equivalence in neighborhoods of large bifurcation supports (which follows from Theorem~\ref{theorem characteristic sets implies moderate equivalence}), but also completely determine the PC-HC family in some neighborhood of the large bifurcation support up to the said equivalence. Our strategy is to prove that any given combinatorial data are realizable by some field (Lemma~\ref{Realization lemma}), and to establish that the mapping sending a family to its combinatorial data is a bijection (Theorem~\ref{theorem about bijection between characteristic sets and vector families}).

\subsection{Realization}

Let a collection of two finite sets $\mathcal{L}_1$ and $\mathcal{L}_2$ on an interval and two finite sets $\mathcal{A}^+$ and $\mathcal{A}^-$ on a circle be given. Each of these sets is marked, i.e., partitioned into equivalence classes of one or two points each. A liaison relation $\sim$ between equivalence classes of the sets $\mathcal{L}_2$ and $\mathcal{A}^+$ is also given.

Introduce a partial order $\le_0$ on the sets $\mathcal{L}_1$ and $\mathcal{L}_2$: for two equivalence classes $\{ \mathfrak{b}_i, \mathfrak{b}_j \}$ and $\{ \mathfrak{b}_r, \mathfrak{b}_s \}$, we say that the first is ``no greater'' than ($\le_0$) the second if the interval $[ \mathfrak{b}_i, \mathfrak{b}_j ]$, bounded by the points of the first class, is entirely contained in the interval $[ \mathfrak{b}_r, \mathfrak{b}_s ]$, bounded by the points of the second class. One-element equivalence classes can only be minimal elements of the partial order $\le_0$. To define the partial order on the circle, fix a point $p \in S^1$ and introduce the order $\le_p$ as follows: the equivalence class $\{ \mathfrak{c}_i, \mathfrak{c}_j \}$ is ``no greater'' than the class $\{ \mathfrak{c}_r, \mathfrak{c}_s \}$ if the arc $\overline{\mathfrak{c}_i \mathfrak{c}_j}$, bounded by the points $\mathfrak{c}_i$ and $\mathfrak{c}_j$ and not containing the point $p$, belongs entirely to the arc $\overline{\mathfrak{c}_r \mathfrak{c}_s}$, bounded by the points $\mathfrak{c}_r$ and $\mathfrak{c}_s$ and not containing $p$. We also assume that an angular coordinate $\varphi_p$ is fixed on the circle on which the set $\mathcal{A}^+$ is defined, so that $\varphi_p(p) = 0$.

\defi \label{realization conditions} We call such a collection of sets $\mathcal{L}_1$, $\mathcal{L}_2$, $\mathcal{A}^+$, $\mathcal{A}^-$ {\itshape realizable with configuration $a_1a_2a_3$} if the following conditions hold:
\begin{itemize}
    \item[ ]
    \begin{enumerate}
        \item two-element equivalence classes of the characteristic sets do not interleave;
        \item points in the liaison relation $\sim$ are arranged consecutively, i.e., there exists a point $p$ on the circle $S^1$ such that for any pairs of points in liaison $\mathfrak{b}_i \sim \mathfrak{c}_r$ and $\mathfrak{b}_j \sim \mathfrak{c}_s$, the implication holds:
        $$
        \zeta(\mathfrak{b}_i) \leqslant \zeta(\mathfrak{b}_j) \quad \Longrightarrow \quad \varphi_p(\mathfrak{a}_r) \leqslant \varphi_p(\mathfrak{a}_s);
        $$
    \end{enumerate}
    \item for configurations with $a_2=0, \; a_3 = 1$:
    \begin{enumerate}
        \item[3.] only maximal elements of the partial orders $\le_0$ and $\le_p$ can be in the liaison relation;
        \item[4.] points in $\mathcal{L}_2$ are in the liaison relation in pairs, i.e., if for some point $\mathfrak{b_i} \in \mathcal{L}_2$ we have $\mathfrak{b}_i \sim \mathfrak{c}_r$, then $\mathfrak{b}_{i-1}$ is in the liaison relation with $\mathfrak{c}_r$ or with $\mathfrak{c}_{r-1}$, or $\mathfrak{b}_{i+1}$ is in the liaison relation with $\mathfrak{c}_r$ or with $\mathfrak{c}_{r+1}$;
    \end{enumerate}
    \item for configurations with $a_2 = 1$ and $a_3 = 1$:
    \begin{enumerate}
        \item[3.] only maximal elements of the partial orders $\le_0$ and $\le_p$ can be in the liaison relation;
        \item[5.] in the set $\mathcal{A}^+$ there exists a point $\mathfrak{a}_{k+1}$ such that conditions 2 and 3 hold for the partial order $\le_{\mathfrak{a}_{k+1}}$;
    \end{enumerate}
    \item for configurations with $a_3 = 0$:
    \begin{enumerate}
        \item[6.] the sets $\mathcal{L}_2$ and $\mathcal{A}^+$ contain no two-element equivalence classes.
    \end{enumerate}
\end{itemize}

\lemma \label{Realization lemma} \textbf{(Realization Lemma)} Let a collection of four sets $\mathcal{L}_1$, $\mathcal{L}_2$, $\mathcal{A}^+$, $\mathcal{A}^-$ realizable with configuration $a_1a_2a_3$ be given. Then there exists an infinitely smooth vector field $v$ of the PC-HC class with configuration $a_1a_2a_3$ whose characteristic sets coincide with the given ones.

Before beginning the proof of this statement, let us recall the formulation of the realization lemma for disks.

\lemma \label{realization lemma for discs} \cite{one-parameter family proof 2} Let a finite marked set $\mathcal{A}$ on the boundary circle $C$ of a disk $D$ be given. Then there exists an infinitely smooth Morse--Smale vector field $v$ on the disk $D$ for which the set $\mathcal{A}$ is the characteristic set.
\vspace{-6pt}

{\itshape Remark:} the field $v$ may be assumed to be orthogonal to the curve $C$.

\textbf{Proof of Lemma~\ref{Realization lemma}.}
We begin by constructing those elements of the phase portrait of the field $v$ that do not depend on the configuration.

\begin{wrapfigure}{r}{0pt}
	\includegraphics[scale=0.27]{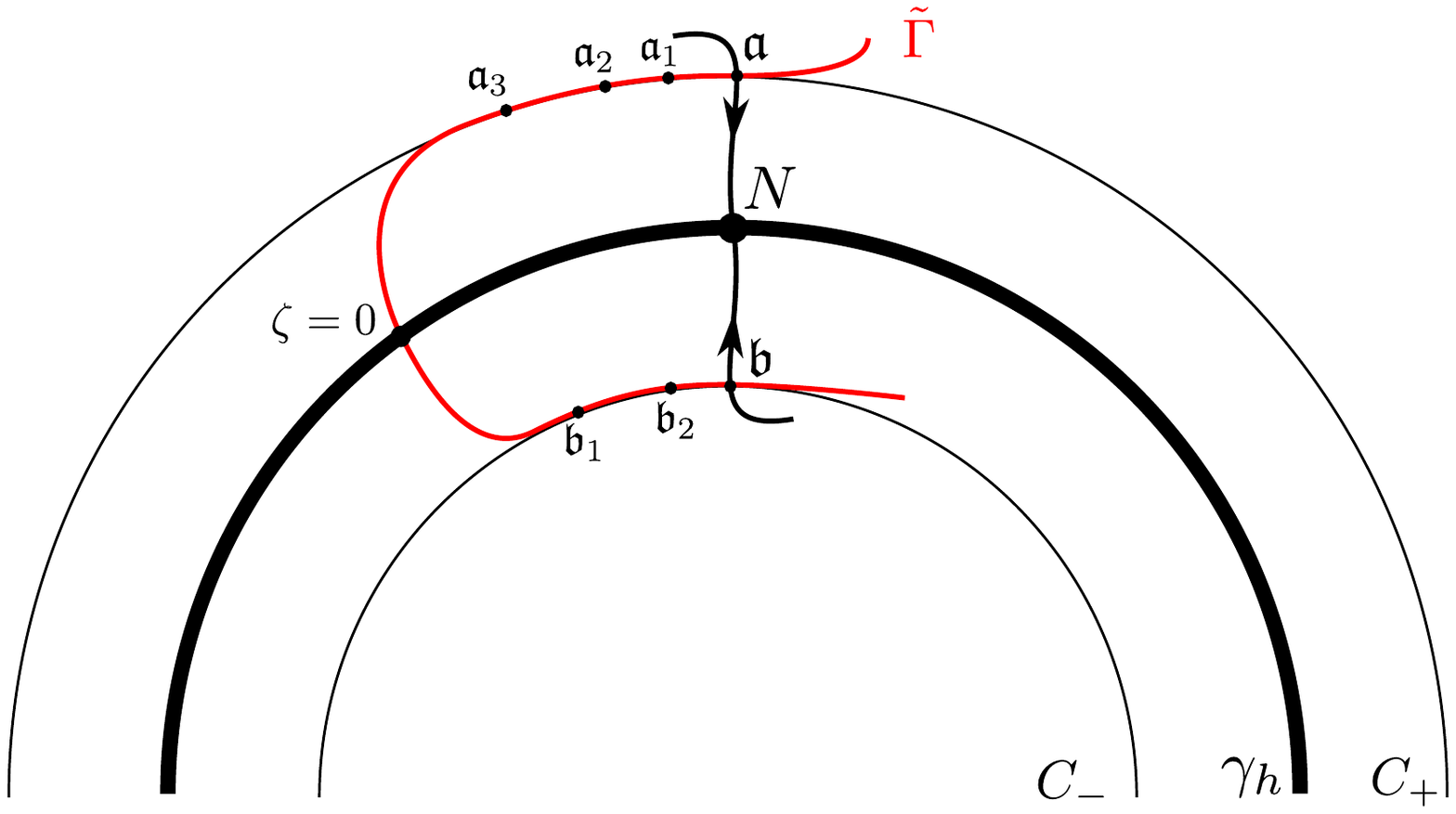}
	\caption{Part of the annulus $U_{HC}$}
	\label{neighbourhood of U_{HC}}
\end{wrapfigure}

Consider a local chart on the phase space $S^2$ with polar coordinates $(r, \phi)$. Let $U_{HC}$ be the annulus $4 \le r \le 6$. Define a vector field $v$ on it so that it has a homoclinic curve $\gamma_h$ of the saddle-node $N$ and no other degeneracies. Let this curve $\gamma_h$ be the circle $\{r = 5\}$, and let the saddle-node $N$ be located at the point $(5, \pi/2)$. For example, the field $v$ in a neighborhood of the curve $\{r=5\}$ can be given by the system
\begin{equation*}
	\left\{
		\begin{aligned}
		\dot  r & = 5-r\\
		\dot \phi & = - \cos^2\left(\frac{\phi}{2} + \frac{\pi}{4}\right).
		\end{aligned}
	\right.
\end{equation*}
Also, in a neighborhood of the boundary circles $C_-$ and $C_+$ of the annulus $U_{HC}$, define the field $v$ so that near $C_-=\{r=4\}$ it coincides with $\partial / \partial r$, and near $C_+=\{r=6\}$ it coincides with $-\partial / \partial r$.

Let $D^+$ be a two-dimensional disk whose boundary $\partial D^+$ coincides with the outer boundary $C_+$ of the annulus $U_{HC}$. To extend the field $v$ to the disk $D^+$, we need to reconcile it with the characteristic set $\mathcal{L}_1$. We may assume that both sets $\mathcal{L}_1$ and $\mathcal{L}_2$ lie on the interval $[-1, 1]$, with the coordinates of points from $\mathcal{L}_2$ strictly less than zero and the coordinates of points from $\mathcal{L}_2$ strictly greater than zero. Draw in the annulus $U_{HC}$ a smooth curve $\tilde \Gamma$, parameterized by the interval $[-1, 1]$ with coordinate $\zeta$, so that each point corresponding to the parameter value $\zeta(\mathfrak{a}_i)$ (where $\mathfrak{a}_i \in \mathcal{L}_1$) lies on the boundary curve $C_+$ of the annulus $U_{HC}$, and each point corresponding to the parameter value $\zeta(\mathfrak{b}_j)$ (where $\mathfrak{b}_j \in \mathcal{L}_2$) from the set $\mathcal{L}_2$ lies on the boundary curve $C_-$. (A possible arrangement of such a curve is shown in Figure~\ref{neighbourhood of U_{HC}}.) Moreover, let the curve $\tilde \Gamma$ be transverse to the field $v$ defined on the annulus $U_{HC}$. This allows us to transfer the characteristic set $\mathcal{L}_1$ to the curve $C_+$. By Lemma~\ref{realization lemma for discs}, the field $v$ can be extended to the disk $D^+$ so that the set $\mathcal{L}_1$ on $\partial D^+ = C_+$ is the characteristic set for the field $\left. v \right|_{D^+}$. The curve $\tilde \Gamma$ also allows us to transfer the set $\mathcal{L}_2$ to the circle $C_-$.

We now proceed to constructing the field $v$ in domains that depend on the configuration. Consider the set $U_{PC}$---the annulus $1 \le r \le 3$ in polar coordinates $r, \phi$. Define the field $v$ on it so that it has no singular points and has one parabolic cycle coinciding with the circle $\{r=2\}$. Such a field in a neighborhood of $\{r=2\}$ can be given by the system
\begin{equation*}
	\left\{
		\begin{aligned}
		\dot  r & = \pm (r - 2)^2\\
		\dot \phi & = \pm 1.
		\end{aligned}
	\right.
\end{equation*}
In this system, the signs $\pm$ are chosen according to the configuration the field should belong to.
If the configuration contains $a_3=1$, then the plus sign is chosen in the expression for $\dot r$; in the case $a_3=0$, the minus sign is chosen. If the configuration contains $a_1=1$, then the plus sign is chosen in the expression for $\dot \phi$; otherwise, the minus sign.
\vspace{-5pt}

Also let the field $v$ coincide with $\pm \partial / \partial r$ in neighborhoods of the boundary circles $C^-=\{r=1\}$ and $C^+=\{r=3\}$. The sign $\pm$ is chosen the same as in the expression for $\dot r$. We assume the field $v$ has no other degeneracies on $U_{PC}$.

Identify the circles $S^1$ on which the sets $\mathcal{A}^{\pm}$ are defined with the boundary circles $C^{\pm}$ of the annulus $U_{PC}$. Let $D^-$ be the two-dimensional disk $\{ r<1 \}$, whose boundary $\partial D^-$ coincides with the inner boundary $C^-$ of the annulus $U_{PC}$. The set $\mathcal{A}^-$ is given on the circle $C^-$. Again by Lemma~\ref{realization lemma for discs}, the field $v$ can be extended to the disk $D^-$ so that the set $\mathcal{A}^-$ on $\partial D^-$ is the characteristic set for the field $v$.

Now consider the annulus $U_b = \{3 \le r \le 4\}$. Its outer boundary is the circle $C_-$, on which the set $\mathcal{L}_2$ is given, and its inner boundary is the circle $C^+$, on which the set $\mathcal{A}^+$ is fixed. We construct the field $v$ on the domain $U_b$ step by step, considering various configurations separately.

\textbf{Configurations 111 and 011}
\vspace{-5pt}

Step 1. By condition 5 of Definition~\ref{realization conditions}, there is a distinguished point $\mathfrak{a}^+_{k+1}$ in the set $\mathcal{A}^+$. By construction, a point $\mathfrak{b}$ is defined on the circle $C_-$---the intersection point of one of the separatrices of the saddle-node $N$ with the curve $C_-$. Connect the points $\mathfrak{a}^+_{k+1}$ and $\mathfrak{b}$ by a simple smooth curve whose interior lies entirely in $\mathrm{Int} \: U_b$. Denote this curve by $\beta_1$ and cut the annulus $U_b$ along it. We obtain a topological rectangle $\overline{U_b}$. On its lower side $C^+$, the set $\mathcal{A}^+$ is given. Take any two points from the set $\mathcal{A}^+$ belonging to the same equivalence class and connect them by a simple smooth arc, transverse to the curve $C^+$ at its endpoints, so that the interior of the arc lies in $\mathrm{Int}\:U_b$. Fix one point on this arc, distinct from its endpoints; it will be a saddle of the field $v$. The two segments of the arc into which the fixed point divides the original arc will serve as stable separatrices. Do the same for all pairs in turn, following the partial order relation, i.e., first draw arcs for all minimal two-element elements of the order.
\vspace{-5pt}

Now carry out the same steps for the points of the set $\mathcal{L}_2$ belonging to two-element equivalence classes. But now the segments of the arc, into which the fixed point divides the original arc connecting $\mathfrak{b}_i$ to $\mathfrak{b}_j$, will serve as unstable separatrices of the saddle.
\vspace{-5pt}

Each arc divides the rectangle $\overline{U_b}$ into topological disks. Denote the one whose boundary contains the curve $\beta_1$ by $\tilde U_b$. (A possible result after Step 1 is shown in Figure~\ref{example of picture of ralization lemma (conf A, B) Step 1}.)

Step 2. Choose one point inside each disk obtained from the subdivision of the rectangle $\overline{U_b}$, except for the disk $\tilde U_b$. For disks whose boundary contains elements of the set $\mathcal{L}_2$, these points will serve as repellers of the field $v$, and for disks whose boundary contains points of the set $A^+$---as attractors of the field $v$. Connect each saddle on the boundary of a disk to the attractor or repeller inside that disk by a simple smooth curve; these curves will be the unstable or stable separatrices of the saddles, respectively. Now every saddle corresponding to a non-maximal element of the partial order has exactly four separatrices: two stable and two unstable.

\begin{figure}
\center{\includegraphics[scale=0.4]{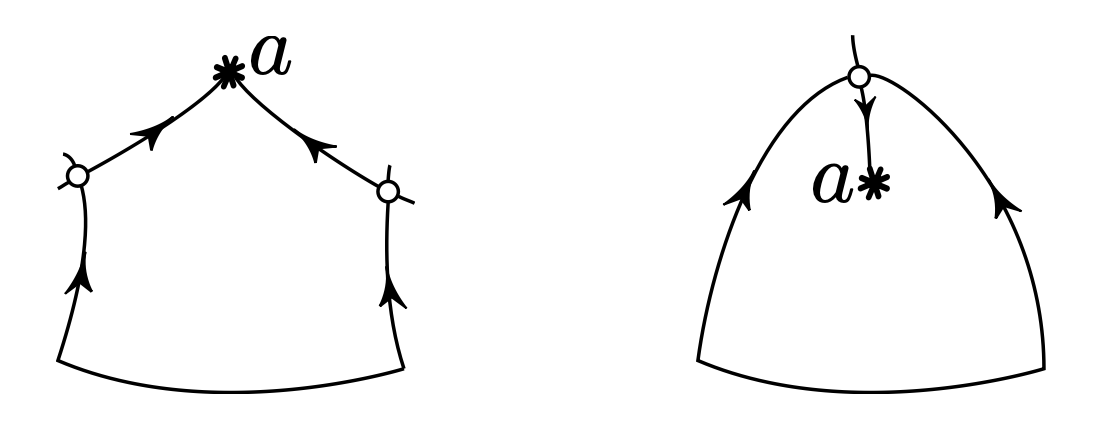}}
\caption{Domains obtained at Step 2, adjacent to the curve $C^+$.}
\label{two types of domains}
\end{figure}

\begin{figure}
\center{\includegraphics[scale=0.4]{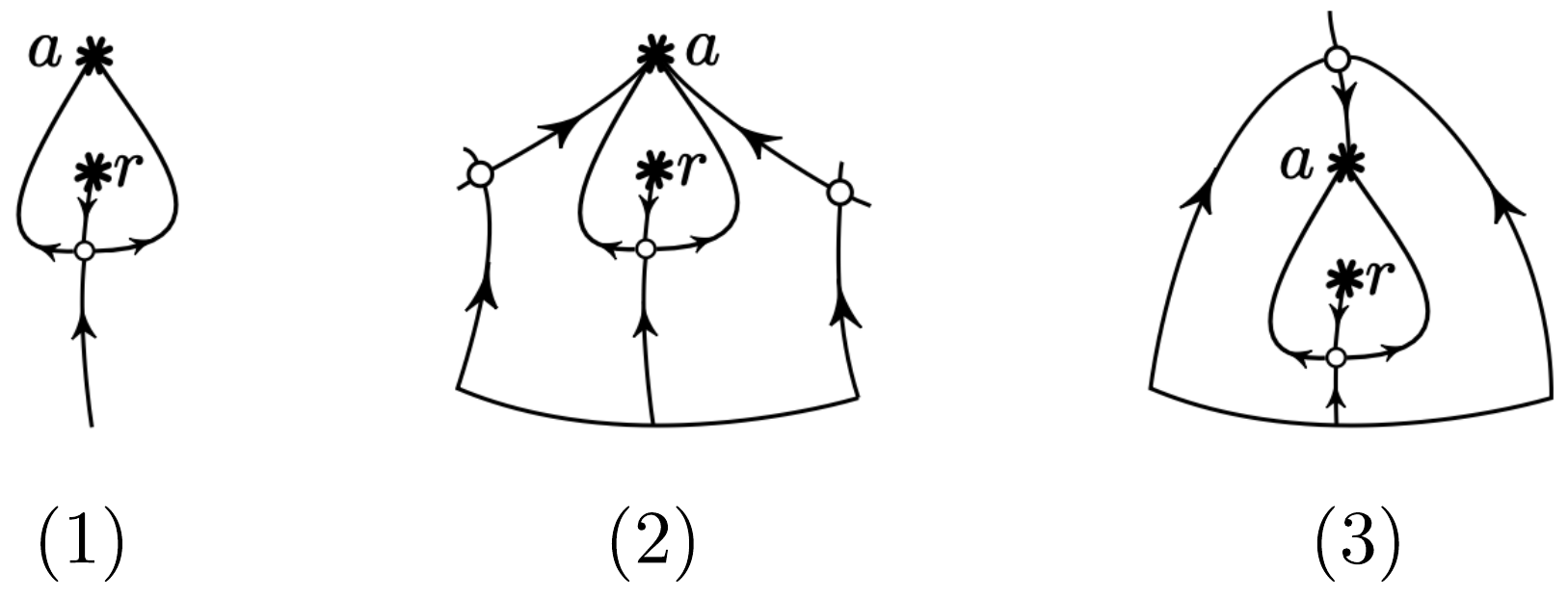}}
\caption{Construction of separatrices corresponding to points of characteristic sets from one-element equivalence classes in $\overline{U_b} \backslash \tilde U_b$.}
\label{step (c) for domains}
\end{figure}

Step 3. The boundaries of all obtained disks, except for the disk $\tilde U_b$, may have one of two forms shown in Figure~\ref{two types of domains}, i.e., they are bounded either by one arc constructed at Step 1, or by two arc segments and curves constructed at Step 2. The directions indicated by arrows in Figure~\ref{two types of domains} can only appear on the boundaries of those domains adjacent to the curve $C^+$. For domains adjacent to the curve $C_-$, the arrows point in the opposite direction. On the curves $C^+$ and $C_-$, take those points of the characteristic sets $\mathcal{A}^+$ and $\mathcal{L}_2$ that lie on the arc segments bounding the domains shown in Figure~\ref{two types of domains}. They all belong to one-element equivalence classes. To those among them that belong to the set $\mathcal{A}^+$, add the graph shown in Figure~\ref{step (c) for domains}.(1) (the asterisk $a$ denotes an attractor, the asterisk $r$---a repeller, and the curves indicate the separatrices of the saddle). The attractor $a$ must coincide with the attractor from the considered domains, as shown in Figures~\ref{step (c) for domains}.(2) and~\ref{step (c) for domains}.(3). To those among these points that belong to the set $\mathcal{L}_2$, add the same graph but with all directions of motion along the edges reversed (and with the attractor $a$ and repeller $r$ swapped), also aligning its position as shown in Figures~\ref{step (c) for domains}.(2) and~\ref{step (c) for domains}.(3).

\begin{figure}
\center{\includegraphics[scale=0.5]{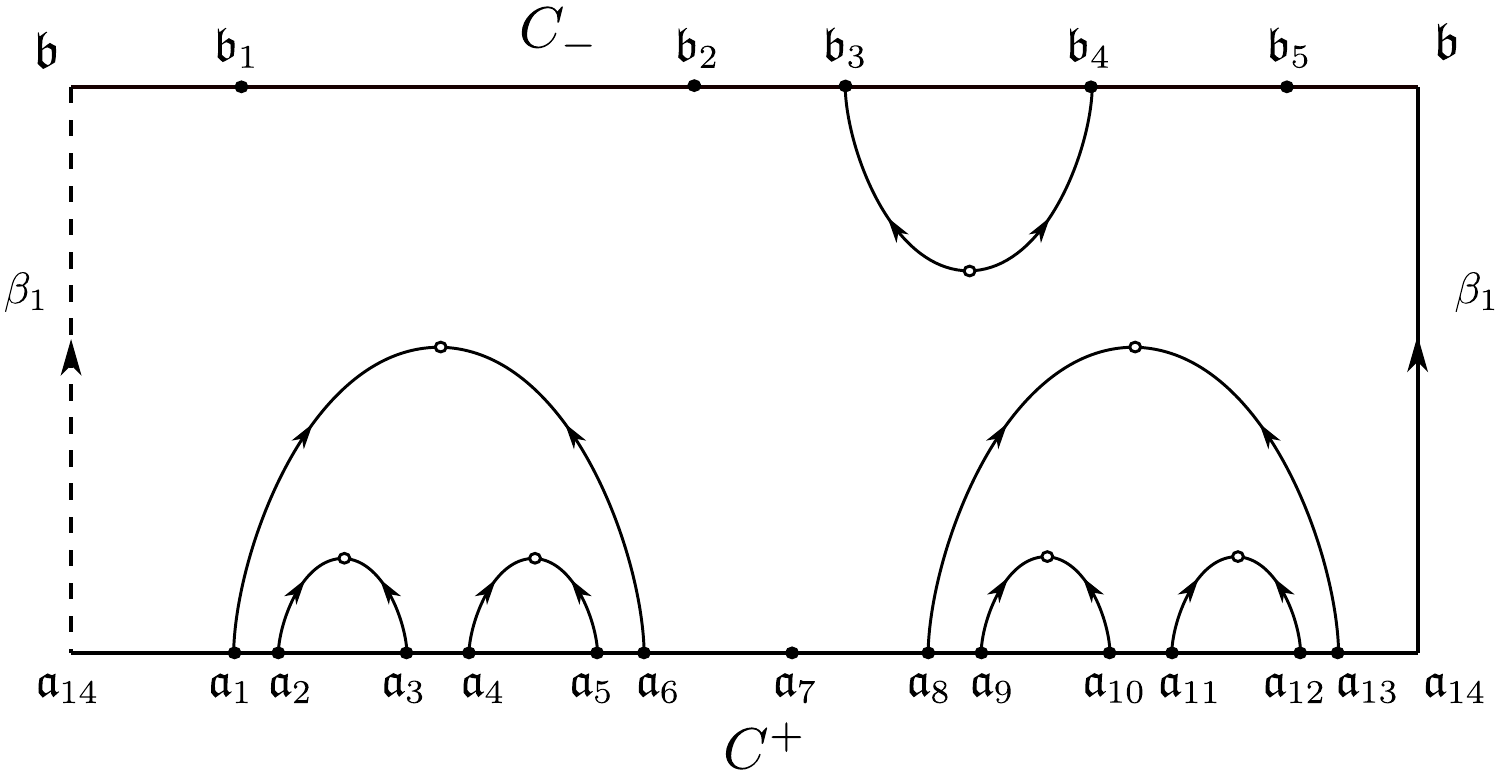}}
\caption{Curves in the rectangle $\overline{U_b}$ obtained after completing Step 1.}
\label{example of picture of ralization lemma (conf A, B) Step 1}
\end{figure}

\begin{figure}
\center{\includegraphics[scale=0.5]{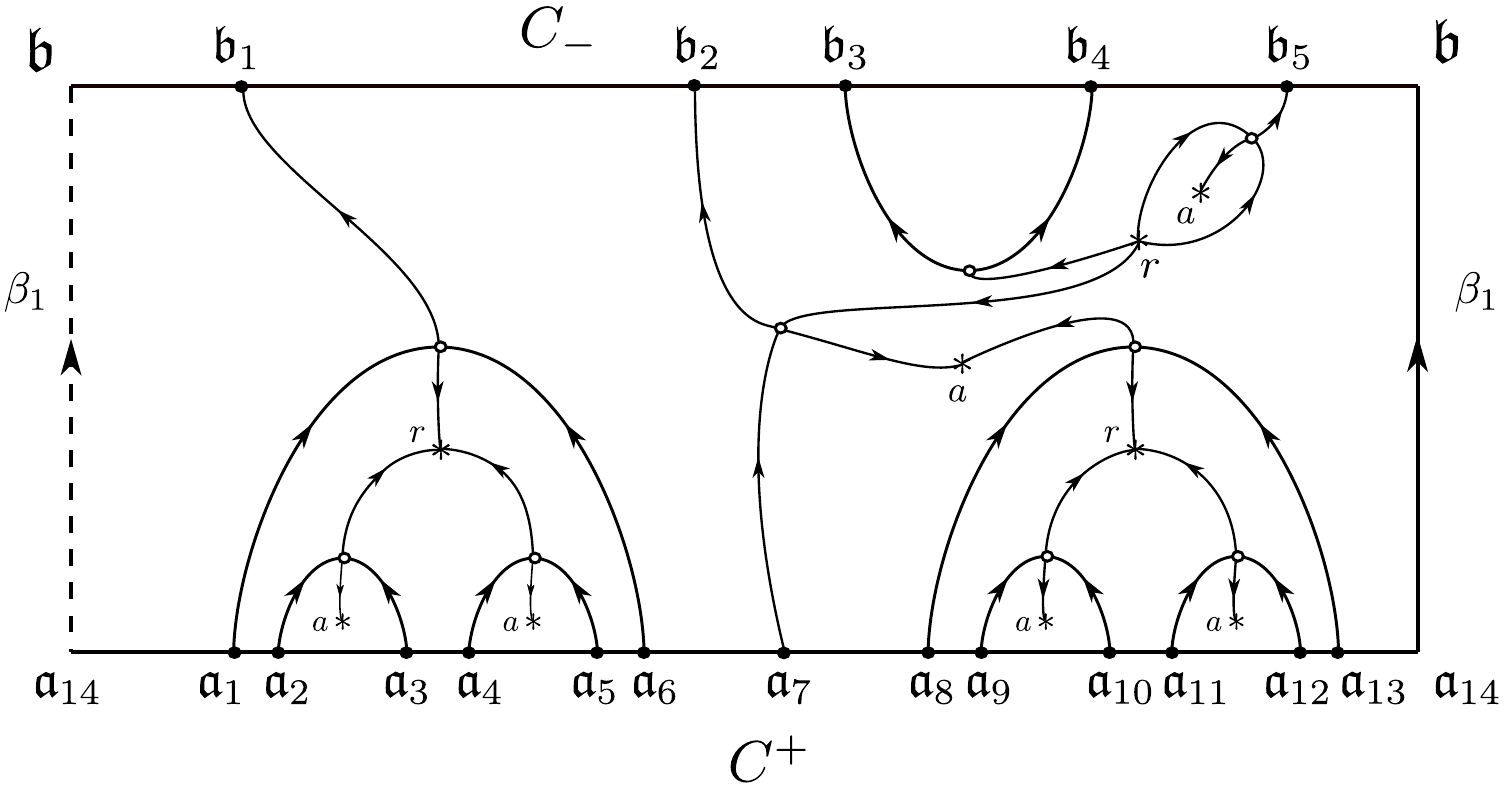}}
\caption{Curves in the rectangle $\overline{U_b}$ obtained after completing Step 4 for configurations 111 and 011.}
\label{example of picture of realization lemma (conf A, B) Step 4}
\end{figure}

Step 4. The set $\tilde U_b$ is homeomorphic to a rectangle by construction. On its boundary lie the points of the characteristic sets $\mathcal{A}^+$ and $\mathcal{L}_2$ belonging to one-element equivalence classes that are maximal elements of the partial order introduced on these characteristic sets. The saddles constructed at Step 1 also lie on the boundary. Draw curves between them, consistent with the liaison relation. Connect each pair in the liaison relation by a smooth simple curve. This is possible by condition 2 of Definition~\ref{realization conditions}. Now, on the curves connecting the one-element equivalence classes from the sets $\mathcal{A}^+$ and $\mathcal{L}_2$, fix points that will serve as saddles, and assign an orientation to the two resulting segments of these curves: the segment connecting the point on the curve $C_-$ to the saddle corresponds to the stable separatrix, and the segment connecting the saddle to the point on the curve $C_+$---to the unstable separatrix (in Figure~\ref{example of picture of realization lemma (conf A, B) Step 4}, the point $\mathfrak{a}_7$ is connected to the point $\mathfrak{b}_2$ by a curve on which a saddle point has been chosen and the corresponding orientations of the segments are indicated). The remaining points related by the liaison relation form pairs in which one point belongs to the set $\mathcal{A}^+$ or $\mathcal{L}_2$, and the other is a saddle constructed at Step 1. Connect the corresponding pairs by simple smooth curves and orient these curves, taking the points on the lower side of the rectangle $\tilde U_b$ as initial points and those on the upper side as terminal points. After this, the rectangle $\tilde U_b$ is subdivided into several small rectangles. For each saddle constructed at this Step 4 (in Figure~\ref{example of picture of realization lemma (conf A, B) Step 4}, such a saddle corresponds to the points $\mathfrak{a}_7$ and $\mathfrak{b}_2$), choose two points in one of the sets bounded by it, making sure they do not lie in the sets bounded by pairs of points in liaison that are forbidden by condition 4 of Definition~\ref{realization conditions}. Designate one of these points as a repeller and the other as an attractor, and connect them to the constructed saddle by curves (which will be the stable and unstable separatrices respectively). To the points of the set $\mathcal{A}^+$ that are not in the liaison relation, add the graph shown in Figure~\ref{step (c) for domains}.(1); to the points of the set $\mathcal{L}_2$ that are also not in the liaison relation, add the same graph but with the edge orientations reversed and with the repeller and attractor swapped (in Figure~\ref{example of picture of realization lemma (conf A, B) Step 4}, this graph is added to the point $\mathfrak{b}_5$). For the remaining saddles, choose one point in their neighborhood, which will serve as a repeller or attractor, connect them to the saddles, and orient the curves accordingly.

This step completes the construction of the separatrices of the field $v$. We can now construct a vector field on the set $U_b$. This can be easily done on the connected components of the complement to the just-constructed separatrices in $U_b$, and the field $v$ can be made infinitely smooth. This completes the construction of the field $v$.

\textbf{Configurations 101 and 001}
\vspace{-5pt}

Step 1. On the circle $C^+$ there is a point $p$ for which condition 2 of Definition~\ref{realization conditions} holds. Moreover, a point $\mathfrak{b}$ is chosen on the circle $C_-$. Connect the points $p$ and $\mathfrak{b}$ by a simple smooth curve and cut the annulus $U_b$ along it, converting it into the rectangle $\overline{U_b}$. Then repeat all actions of Step 1 for configurations 111 and 011.

Steps 2 and 3 are carried over from the construction for configurations 111 and 011 without changes.

Step 4 has a single modification: in the rectangle containing the point $\mathfrak{b}$ on its boundary, after adding the repeller point, we connect it to the point $\mathfrak{b}$ by a smooth curve and call this curve $\beta_1$.

\textbf{Configurations 100 and 000}
\vspace{-5pt}

\begin{figure}
\center{\includegraphics[scale=0.5]{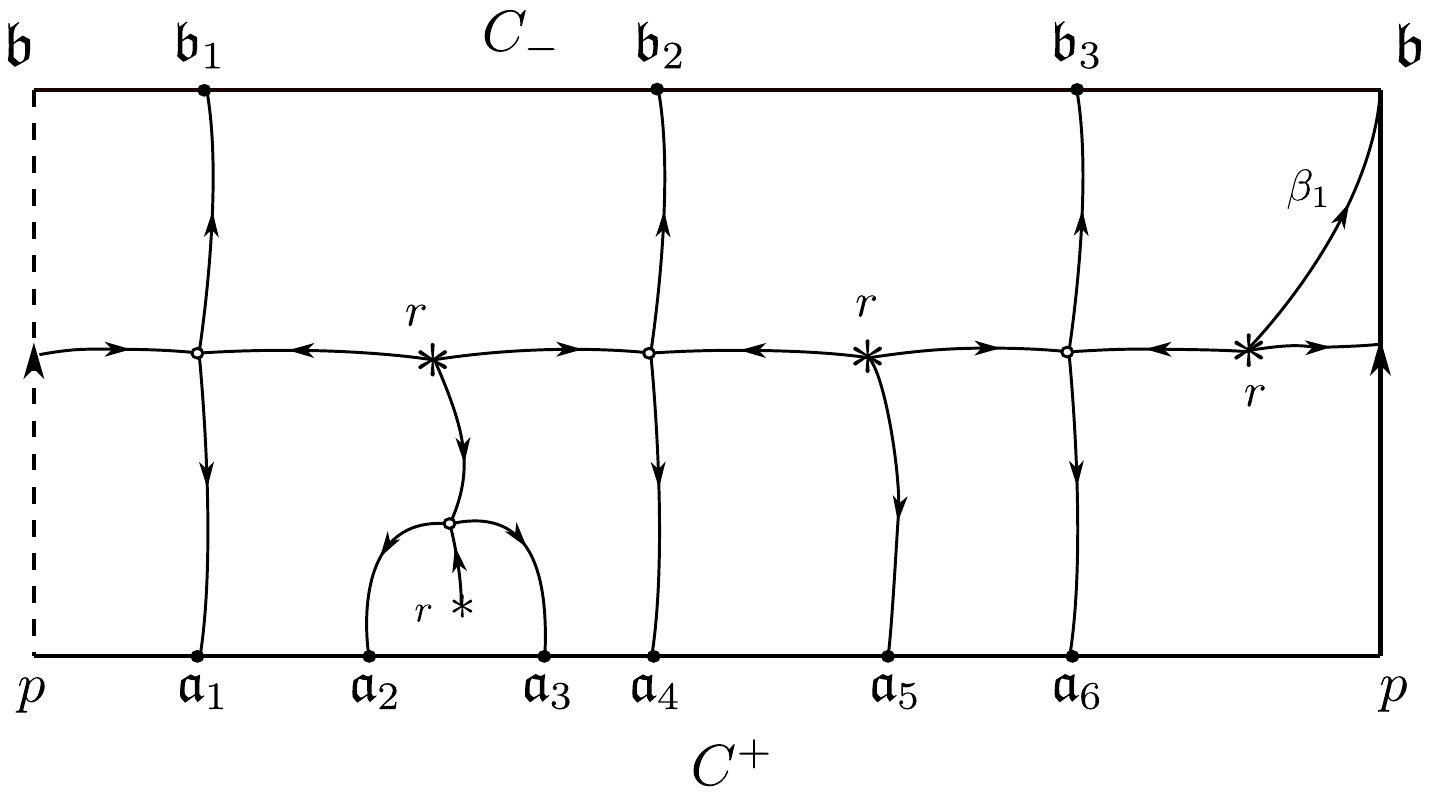}}
\caption{Curves in the rectangle $\overline{U_b}$ obtained after completing Step 4 for configurations 100 and 000.}
\label{example of picture of ralization lemma (conf E, F) Step 4}
\end{figure}

Steps 1, 2, and 3 are carried out analogously to the case of configurations 101 and 001.

Step 4. Connect the one-element equivalence classes in the liaison relation by smooth curves. (Note that configurations 100 and 000 have no two-element equivalence classes in the liaison relation, by condition 6 of Definition~\ref{realization conditions}.) On each such curve, choose one point---it will correspond to a saddle. Orient the segments of the curves emanating from it so that they correspond to the unstable separatrices. The drawn curves subdivide the rectangle $\tilde U_b$ into small rectangles. In each of them, choose a point---this will be a repeller. Connect the remaining saddles and points of the characteristic sets to these repellers. Also draw a curve connecting the point $\mathfrak{b}$ to the repeller lying in the same rectangle as $\mathfrak{b}$. Denote this curve by $\beta_1$. (An example is shown in Figure~\ref{example of picture of ralization lemma (conf E, F) Step 4}.)

Thus, the construction of the field $v$ is complete. \qed

\subsection{Classification of Large Bifurcation Supports}

Combining the obtained results, we can state a general assertion.

Consider the correspondence $\wp$ between two sets:
\begin{enumerate}
	\item[$\mathcal{G}$:] the set obtained by factoring the set of all generic families of the PC-HC class by the equivalence relation defined as follows: $V_1$ is equivalent to $V_2$ if $V_1$ and $V_2$ are moderately equivalent in some neighborhoods of their large bifurcation supports;
	\item[$\mathcal{M}$:] the set consisting of pairs of the form $(a_1a_2a_3, b)$, where $a_1a_2a_3$ is a configuration and $b$ is an equivalence class of realizable characteristic sets with configuration $a_1a_2a_3$ in the sense of Definition~\ref{realization conditions}.
\end{enumerate}

Let the correspondence $\wp$ associate to the equivalence class of a family $V$ (under the relation defining the quotient set $\mathcal{G}$) its configuration and the equivalence class of its characteristic sets.

\thm \label{theorem about bijection between characteristic sets and vector families} The correspondence $\wp$ is a bijection.

\textbf{Proof.} The correspondence $\wp$ is a map, since if two families $V_1$ and $V_2$ are moderately equivalent in neighborhoods of their large bifurcation supports, then their characteristic sets are equivalent and the families have the same configurations; hence the correspondence $\wp$ associates to each element $g \in \mathcal{G}$ exactly one element $m \in \mathcal{M}$.

By Lemma~\ref{Realization lemma}, the mapping $\wp$ is surjective.

By Theorem~\ref{theorem characteristic sets implies moderate equivalence}, the mapping $\wp$ is injective. \qed

\begin{Theorem}
The two large bifurcation supports of PC-HC families $V$ and $\tilde V$ with equivalent configurations and characteristic sets are mapped to each other by a homeomorphism of the sphere.
\end{Theorem}

\textbf{Proof.} By Theorem~\ref{theorem about bijection between characteristic sets and vector families}, there exists a moderate equivalence mapping $\mathbf{H}$ of the families $V$ and $\tilde V$ in neighborhoods $U$ and $\tilde U$ of their large bifurcation supports. At the zero value of the parameter, this mapping $\mathbf{H}_0$ is a homeomorphism of the neighborhoods $U$ and $\tilde U$ sending the set LBS($V$) to the set LBS($\tilde V$). We show that the mapping $\mathbf{H}_0$ can be extended to a homeomorphism of the entire sphere.

If the large bifurcation support LBS($V$) is disconnected, it consists of two connected components. The first component contains the saddle-node $N$, its homoclinic curve $\gamma_h$, as well as the saddles $A_i$ and $B_j$ and their separatrices $a_i$ and $b_j$. Note that this set is topologically a wedge of finitely many circles and intervals with a marked point $N$. Hence the neighborhood of this set, consisting of the domains $U_N$, $U_{\gamma_h}$, $U_{A_i}$, $U_{a_i}$, $U_{B_j}$, $U_{b_j}$ constructed in Section~\ref{subsection about choice of neighbourhood of LBS}, is a topological space obtained by gluing together topological annuli (neighborhoods $U_{a_i} \cup U_{A_i} \cup U_{a_j} \cup U_N$ of pairs of separatrices $a_i$ and $a_j$ emanating from the same saddle $A_i$, and the neighborhood $U_{\gamma_h} \cup U_N$ of the curve $\gamma_h$) and topological disks (neighborhoods of the form $U_{a_i} \cup U_N$ of single separatrices $a_i$) along the set $U_N$. The complement of such a set is a finite collection of disks $\mathcal{D}_1, \ldots, \mathcal{D}_e$. Similarly, for the neighborhood of the second connected component of the large bifurcation support LBS($V$), consisting of the parabolic cycle $\gamma$ and the separatrices $c_i$ and $d_j$ with saddles $C_i$ and $D_j$, one can show that this neighborhood is also represented as a gluing of finitely many topological annuli and disks. Hence the complement of this neighborhood is also a finite collection of topological disks $\mathcal{D}^1, \ldots, \mathcal{D}^f$. The second connected component of the set LBS($V$) lies in one of the disks $\mathcal{D}_g$, and the first connected component lies in one of the disks $\mathcal{D}^h$. Hence, among the pairwise intersections of the disks $\mathcal{D}_i \cap \mathcal{D}^j$, the only non-trivial one is the intersection of the specified pair: $\mathcal{D}_g \cap \mathcal{D}^h$. This intersection is a topological annulus $R$. Thus the set $S^2 \setminus U$ is the disjoint union of a topological annulus and a finite number of topological disks.

For the case of a connected large bifurcation support LBS($V$), the argument is analogous. The only addition is that neighborhoods of the form $U_{b_i} \cup U_{B_i} \cup U_{c_j}$ of separatrices $b_i$ and $c_j$ emanating from the same saddle $B_i$, or the neighborhood $U_{\beta_2}$ of the stable separatrix $\beta_2$ of the saddle-node $N$, lie in the topological annulus $R$, connecting its boundary circles. And a topological annulus from which neighborhoods of curves connecting its boundary circles are removed is homeomorphic to a finite collection of disks. Hence the set $S^2 \setminus U$ consists only of topological disks.

By the construction in Section~\ref{subsection about moderate equivalence in neighbourhoods of LBS}, the mapping $\mathbf{H}_0$ is defined in the neighborhood $U$ of the large bifurcation support and extends to its boundary $\partial U$. Since any homeomorphism of the boundary circles of two disks can be extended to a homeomorphism between these disks, and the complement of the set $U$ is a union of finitely many disks, we conclude that the mapping $\mathbf{H}_0$ extends to a homeomorphism on each of these disks, and therefore to a homeomorphism of the entire sphere. \qed

\textbf{Proof of Theorem~\ref{theorem about equivalence of bifurcation diagrams}.} If two generic PC-HC families $V$ and $\tilde V$ have the same configurations and equivalent characteristic sets, then by Theorem~\ref{theorem characteristic sets implies moderate equivalence} there exists a moderate equivalence mapping $\mathbf{H}$ in neighborhoods of their large bifurcation supports. If the field $v_{\varepsilon_0, \lambda_0}$ in the family $V$ is not structurally stable, i.e., it has either a parabolic cycle, or a saddle-node, or a saddle connection between saddles $C_i$ and $D_j$, then the corresponding field $\tilde v_{h(\varepsilon_0, \lambda_0)}$ from the family $\tilde V$ will have the same degeneracy, since the mapping $\mathbf{H}_{\varepsilon_0, \lambda_0}$ is an orbital topological equivalence and preserves the said degeneracies. Thus the mapping $h = \left. \mathbf{H} \right|_B$ sends the bifurcation values of the parameter of the family $V$ to the bifurcation values of the parameter of the family $\tilde V$; and since $h$ itself is a homeomorphism of the parameter spaces, it provides a homeomorphism between the bifurcation diagrams of the families $V$ and $\tilde V$. \qed

\section{Bifurcation Scenario}
\label{chapter about bifurcation diagrams}

\subsection{Separatrix Connections}

In this section we describe the qualitative picture of bifurcations in generic PC-HC families. The key phenomenon is the appearance of an infinite sequence of saddle connections---which we call sparkling saddle connections---upon the breakdown of the parabolic cycle $\gamma$, i.e., as the parameter $\varepsilon \to 0+$. Our goal is to show that the order in which these connections appear is completely determined by the characteristic sets $A^+$ and $A^-$ introduced in Section~\ref{subsection about characteristic sets}.

Recall that we introduced the parameters $(\varepsilon, \lambda)$ so that the parameter $\varepsilon$ corresponds to the bifurcation of the parabolic cycle $\gamma$, and the parameter $\lambda$ to the bifurcation of the saddle-node $N$. For $\varepsilon<0$, the cycle splits into two hyperbolic limit cycles, while for $\varepsilon>0$ the cycle disappears. Similarly, negative values of the parameter $\lambda$ correspond to the splitting of the saddle-node into a saddle and a topological node, while positive values correspond to the disappearance of the saddle-node and the birth of a stable hyperbolic limit cycle.

Consider the bifurcation scenario in families with configurations with $a_2 = 1$ and $a_3 = 1$. Upon the disappearance of the limit cycle $\gamma$, the separatrices $d_i$ tending to it in positive time may connect with the separatrices $c_j$ tending to the cycle in negative time, or with the unstable separatrix $\beta_1$ of the saddle-node. As the parameter $\varepsilon$ changes, these saddle connections may break and form new ones with other separatrices. We describe a sufficient condition for the appearance of such saddle connections.

We derive the saddle connection equation.
Let, as in Section~\ref{subsection about characteristic sets}, the points $b^-$ and $b^+$ be fixed on the transversal $\Gamma_2$ and the time coordinates $T^-$ and $T^+$ be defined relative to them (i.e., the quantity $T^{\pm}(x)$ corresponds to the ``time of motion'' along the curve $\Gamma_2$ from the point $b^{\pm}$ to the point $x$; the formal construction is described in~\cite{one-parameter family proof 2}). In these coordinates, the Poincaré map $\mathcal{P}$ coincides with the unit time shift. They are defined for all values of the parameter $\mu = (\varepsilon, \lambda)$ from a neighborhood of the origin. As before, $\mathfrak{c}_i = c_i \cap [b^+, \mathcal{P}(b^+))$ and $\mathfrak{d}_j = d_j \cap [b^-, \mathcal{P}(b^-))$, and moreover $\mathfrak{c}_{k+1} = \beta_1 \cap [b^+, \mathcal{P}(b^+))$.

Let
$$
\tau(\varepsilon, \lambda) = T^-_{\varepsilon, \lambda}(b^+).
$$
This quantity is well defined for $\varepsilon > 0$ and tends to $+\infty$ as $\varepsilon \to 0+$. Informally, $\tau(\varepsilon, \lambda)$ is the ``time'' it takes for the point $b^-$ to reach the point $b^+$ while moving along the curve $\Gamma_2$ in accordance with the Poincaré map $\mathcal{P}_{\varepsilon, \lambda}$.

A saddle connection appears if the image of the intersection point of the separatrix $d_j$ with the semi-interval $[b^-, \mathcal{P}(b^-))$ (the point $\mathfrak{d}_j$) under some iterate of the map $\mathcal{P}_{\varepsilon, \lambda}$ coincides with the intersection point of the separatrix $c_i$ with the semi-interval $[b^+, \mathcal{P}(b^+))$ (the point $\mathfrak{c}_i$) for some $i$ and $j$:
$$
\mathcal{P}^{m}_{\varepsilon, \lambda}(\mathfrak{d}_j) = \mathfrak{c}_i, \quad \mbox{where} \;\; m \in \mathbb{N}.
$$
Taking into account the fact that in the coordinates $T^{\pm}$ the Poincaré map $\mathcal{P}_{\varepsilon, \lambda}$ is a shift, we can rewrite the saddle connection equation as:
\begin{equation}
\label{connection equation}
T^+(\mathfrak{c}_i) - T^-(\mathfrak{d}_j) = \tau(\varepsilon, \lambda) + m, \quad m \in \mathbb{Z}.
\end{equation}

The properties of the roots of the saddle connection equation in the one-dimensional case (without dependence on the parameter $\lambda$) are studied in~\cite{one-parameter family proof 2}. We use these results to describe the behavior of the solutions of equation~(\ref{connection equation}) for fixed values of the parameter $\lambda = \lambda_0$.

\prop \label{existance of roots of connection equation} \cite{one-parameter family proof 2} For fixed values $\lambda = \lambda_0$, equation~(\ref{connection equation}) has exactly one solution $\varepsilon = \varepsilon_{mij}$ for sufficiently large $m$.

Consider the value $\varepsilon_{mij}$ corresponding to the solution of equation~(\ref{connection equation}) for $\mathfrak{c}_i$, $\mathfrak{d}_j$ with fixed number $m$. This value corresponds to the appearance of a saddle connection between the separatrices $c_i$ and $d_j$.
The index $i$ can take values from 0 to $k+1$, the index $j$ takes values from 0 to $n$ for $\lambda_0 > 0$. For non-positive values of $\lambda_0$, the separatrix $\beta_1$ of the saddle-node $N$ persists and also forms connections with the separatrices $d_i$, which is reflected by the index value $i = k+1$.

Consider the order in which the bifurcation values of the parameter $\varepsilon$ (i.e., the numbers $\varepsilon_{mij}$ corresponding to saddle connections) appear as the parameter approaches zero.

\prop \label{proposition about order of bifurcation parameter} \cite{one-parameter family proof 2} For a fixed value $\lambda = \lambda_0$, the roots of equation~(\ref{connection equation}) are arranged in such an order that for large $m \in \mathbb{N}$ and for any indices $i,\: i' \in \{1, \ldots, k+1\}$ and $j,\: j' \in \{1, \ldots, n\}$, the relation $\varepsilon_{(m+1)ij} < \varepsilon_{mij}$ holds. (For non-positive values of $\lambda_0$, the indices $i$ and $i'$ take values from 0 to $k+1$.) Moreover, for sufficiently large $m$, the order on the set $\{\varepsilon_{mij}\}_{ij}$ coincides with the order on the set $\{\mathfrak{a}^+_i - \mathfrak{a}^-_j\}_{ij}$, i.e., the following implication holds: if for some $i$, $i'$, $j$, and $j'$ we have $\mathfrak{a}^+_i-\mathfrak{a}^-_j < \mathfrak{a}^+_{i'}-\mathfrak{a}^-_{j'}$, then $\varepsilon_{mij}<\varepsilon_{mi'j'}$.

\prop \label{proposition about shift in bifurcation paratameter} \cite{one-parameter family proof 2} For fixed values of $\lambda$, shifting the points $b^+$ and $b^-$ leads to a shift in the $m$-indexing of the set $\{\varepsilon_{mij}\}$; a shift $m \mapsto m + M$ by any $M \in \mathbb{N}$ can be achieved by an appropriate choice of $b^+$ and $b^-$.

Thus, as the parameter $\varepsilon$ tends to zero, saddle connections appear, and the sequence of their appearance coincides with the sequence of elements in the set $\{\mathfrak{a}^+_i - \mathfrak{a}^-_j\}_{ij}$. If at parameter values $(\varepsilon_{m'ij}, \lambda_0)$ some saddle connection makes $m$ turns in a neighborhood of the cycle $\gamma$ (with the cycle itself absent), then at parameter values $(\varepsilon_{(m'+1)ij}, \lambda_0)$ the saddle connection of the same $\alpha$- and $\omega$-limit sets makes $m+1$ turns in the same neighborhood of the limit cycle. This is a cyclically repeating scenario (Dehn twist).

As the parameter $\lambda$ changes in a neighborhood of the saddle-node $N$, the usual saddle-node bifurcation takes place. For non-positive values of $\lambda$, one stable separatrix $\beta_1$ is preserved, whose $\omega$-limit set is either the saddle-node $N$ or the saddle arising from the splitting of the saddle-node into a saddle and a node. If the family $V$ has configurations with $a_2=1$, then the separatrix $\beta_1$ also forms sparkling saddle connections with the separatrices $d_i$ tending to the limit cycle in forward time, as described by the argument above. If the family $V$ has other configurations, then the separatrix $\beta_1$ does not participate in forming sparkling saddle connections. For any configuration, when the parameter $\lambda$ takes positive values, the saddle-node $N$ disappears and instead of the homoclinic curve $\gamma_h$ a hyperbolic limit cycle is born. All separatrices $a_i$ and $b_j$ that previously tended to the saddle-node $N$ begin to wind onto the new limit cycle $\gamma_h$; the separatrix $\beta_1$ disappears, and the bifurcations in the neighborhood of the parabolic cycle $\gamma$ involve only the separatrices $c_i$ and $d_j$.

From the property described in Proposition~\ref{existance of roots of connection equation}, it follows that the set $G$ belongs to the large bifurcation support:

\textbf{Proof of Proposition~\ref{Set G in SepV}.} Consider a point $g$ on the transversal $\Gamma_2$ belonging to the set $G$. A separatrix passes through this point if for some $j$ the equation
$$
T^+(g) - T^-(\mathfrak{d}_j) = \tau(\varepsilon, \lambda) + m, \quad m \in \mathbb{Z}
$$
is satisfied.
By Propositions~\ref{existance of roots of connection equation} and~\ref{proposition about order of bifurcation parameter}, this equation has a solution $\varepsilon_m(g)$ for any fixed sufficiently large $m$. Moreover, these solutions decrease monotonically to zero as a function of $m$. Hence, in any neighborhood of the origin $(\mathbb{R}, 0)$ of the parameter $\varepsilon$, there exists a value $\varepsilon_m(g)$ such that the corresponding vector field has a separatrix of one of the saddles $D$ passing through the point $g$. Therefore $g \in \overline{Sep \; V}$. \qed

\subsection{Simple Bifurcation Diagram}
\label{subsection with bifurcation diagrams}

\begin{figure}
\center{\includegraphics[scale=0.6]{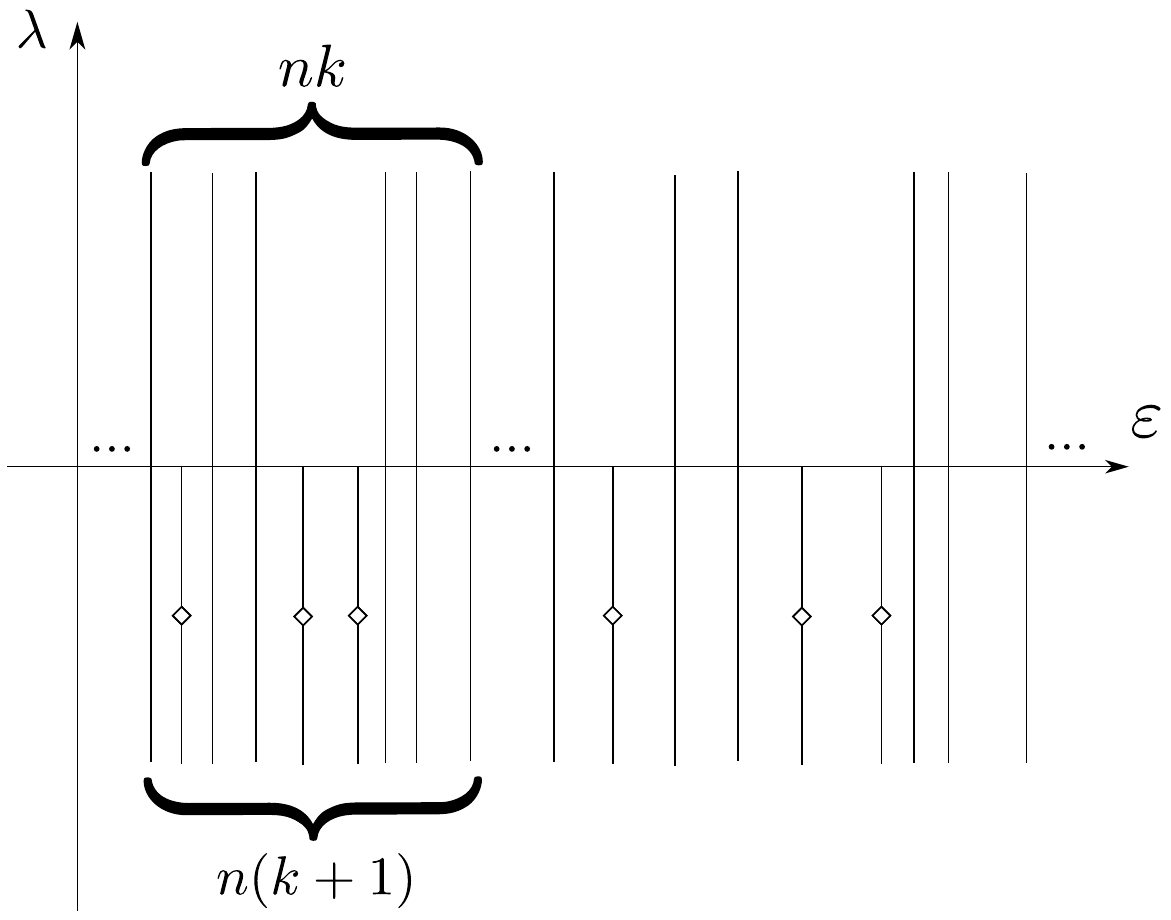}}
\caption{Simple bifurcation diagram of a PC-HC family with configurations with $a_2 = 1$. The small square marks the germ of a curve corresponding to the appearance of a saddle connection involving the separatrix $\beta_1$ of the saddle-node $N$. Here $n$ is the number of separatrices $d_j$ and $k$ is the number of separatrices $c_i$.}
\label{bifurcation diagram for A B}
\end{figure}

In addition to the results obtained in the previous section, we need to consider the dependence of the formation of saddle connections on the parameter $\lambda$.

First consider the family $V$ with configurations with $a_2 = 1$.

For $\lambda < 0$, the saddle-node $N(\varepsilon, 0)$ splits into the saddle $S_0(\varepsilon, \lambda)$ and the node $N_0(\varepsilon, \lambda)$. The separatrix $\beta_1(\varepsilon, \lambda)$, depending smoothly on the parameters, tends to the saddle $S_0(\varepsilon, \lambda)$; it intersects the interval $[b^+, \mathcal{P}(b^+))$ of the transversal $\Gamma_2$ at the point $\mathfrak{a}^+_{k+1}(\varepsilon, \lambda)$. The separatrices $c_i$ tending to the limit cycle in backward time intersect the interval $[b^+, \mathcal{P}(b^+))$ at the points $\mathfrak{a}^+_{i}(\varepsilon, \lambda)$, where $i \le k$. As $(\varepsilon, \lambda) \to 0$, the points $\mathfrak{a}^+_{i}(\varepsilon, \lambda)$ tend to the points $\mathfrak{a}^+_{i}$ of the characteristic set $\mathcal{A}^+$. Similarly, the separatrices $d_j$ tending to the limit cycle $\gamma$ in positive time intersect the interval $[b^-, \mathcal{P}(b^-))$ of the transversal $\Gamma_2$ at the points $\mathfrak{a}^-_{j}(\varepsilon, \lambda)$. The condition for the appearance of a saddle connection between the separatrices $c_i$ and $d_j$ can then be written as
$$
a^+_i(\varepsilon, \lambda) = \mathcal{P}^m_{\varepsilon, \lambda}(a^-_j(\varepsilon, \lambda)),
$$
where $\mathcal{P}_{\varepsilon, \lambda}$ is the Poincaré map in a neighborhood of the cycle $\gamma$. This is equivalent to equation~(\ref{connection equation}). By Proposition~\ref{existance of roots of connection equation}, for large $m$ it has a unique solution $\varepsilon_{mij}$ for each fixed $\lambda$, i.e., it is a function of $\lambda$: $\varepsilon = \psi_{mij}(\lambda)$. Since all terms of equation~(\ref{connection equation}) are smooth in $\lambda$ (for $\lambda<0$), the function $\psi_{mij}(\lambda)$ is also smooth in $\lambda$ and tends to $\varepsilon_{mij}$ as $\lambda \to 0-$.

For $\lambda > 0$, the separatrix $\beta_1(\varepsilon, \lambda)$ disappears and only the points corresponding to the separatrices $c_i(\varepsilon, \lambda)$ remain in the set $\{\mathfrak{a}^+_{i}(\varepsilon, \lambda)\}$. An analogous argument shows that the solutions $\varepsilon = \psi_{mij}(\lambda)$ of equation~(\ref{connection equation}) depend smoothly on the parameter $\lambda$ and tend to $\varepsilon_{mij}$ as $\lambda \to 0+$.

\begin{figure}
\center{\includegraphics[scale=0.6]{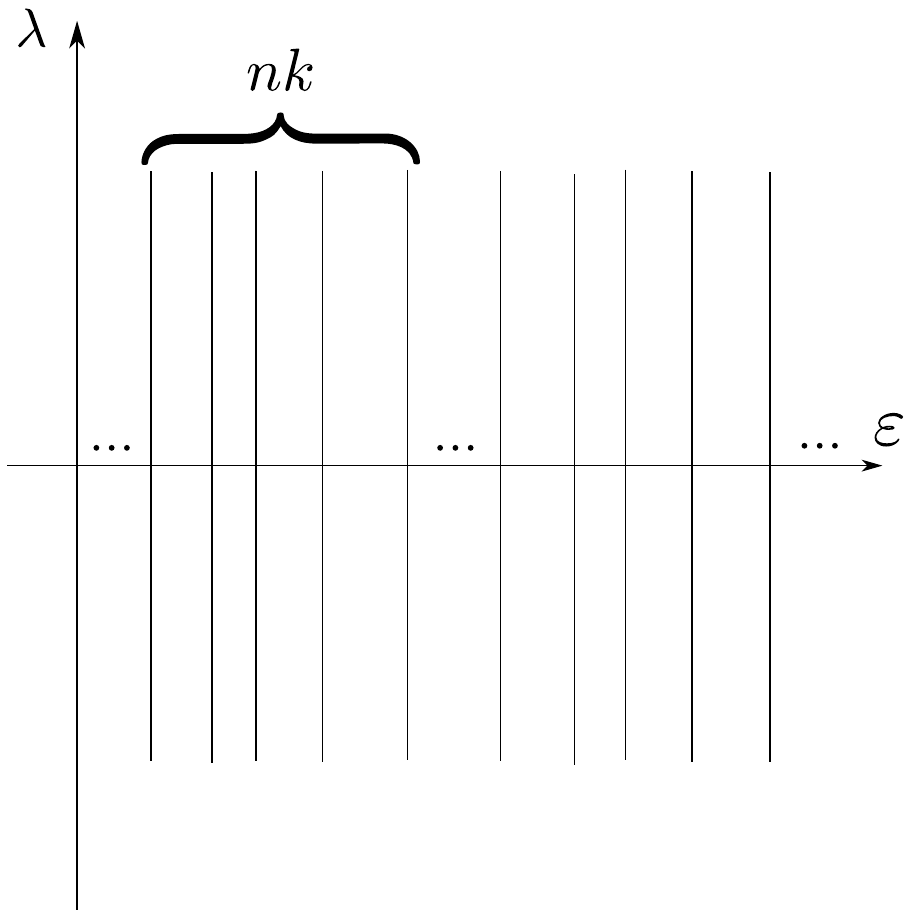}}
\caption{Simple bifurcation diagram of a PC-HC family with configurations with $a_2 = 0$. Here $n$ is the number of separatrices $d_j$ and $k$ is the number of separatrices $c_i$.}
\label{bifurcation diagram for C D E F}
\end{figure}

Thus, the simple bifurcation diagram of a PC-HC family for configurations with $a_2=1$ is the union of the germs of the coordinate axes $\lambda=0$ and $\varepsilon=0$, together with a sequence of germs of planar graphs for $\varepsilon > 0$: near the point $(\varepsilon_s, 0)$, the bifurcation diagram is homeomorphic to the germ of a graph at a vertex from which either four edges emanate (if the point $(\varepsilon_s, 0)$ corresponds to a saddle connection between saddles $C_i$ and $D_j$), or three edges (if the point $(\varepsilon_s, 0)$ corresponds to a saddle connection between the stable separatrix $\beta_1$ of the saddle-node $N$ and saddles $D_j$). In this case, for every one vertex with three edges there are $k$ vertices with four edges. The sequence in which the vertices alternate repeats with period $n(k+1)$ vertices (the germs of the bifurcation diagram in neighborhoods of the points $(\varepsilon_s, 0)$ and $(\varepsilon_{s+n(k+1)}, \lambda)$ are homeomorphic). Such a diagram for the case $n=3$ and $k=2$ is shown in Figure~\ref{bifurcation diagram for A B}.

Now consider the family $V$ with configurations with $a_2 = 0$.

In this case, the separatrix $\beta_1$ does not wind off the limit cycle $\gamma$ and does not form sparkling connections with the separatrices $d_j$. The bifurcation in the neighborhood of the saddle-node $N$ proceeds in the same way as in the case of configurations with $a_2 = 1$. In the neighborhood of the parabolic cycle $\gamma$, the condition for the appearance of sparkling saddle connections, as before, is described by equation~(\ref{connection equation}) for the points of the sets $\mathcal{A}^+$ and $\mathcal{A}^-$, where the set $\mathcal{A}^+$ has exactly $k$ elements $\mathfrak{a}^+_i$, each corresponding to the separatrix $c_i$. The solutions of this equation are smooth in the parameter $\lambda$.

We have thus found that the simple bifurcation diagram of a PC-HC family for configurations with $a_2 = 0$ is the union of the germs of the coordinate axes $\lambda=0$ and $\varepsilon=0$, together with a sequence of germs of planar graphs for $\varepsilon > 0$: near the point $(\varepsilon_s, 0)$, the bifurcation diagram is homeomorphic to the germ of a graph at a vertex from which four edges emanate. In other words, the bifurcation diagram in this case is homeomorphic to the diagram of a direct sum of bifurcations. Such a diagram for the case $n=1$ and $k=5$ is shown in Figure~\ref{bifurcation diagram for C D E F}.

\section*{Acknowledgments}
The author expresses deep gratitude to his supervisor, Yulij Sergeevich Ilyashenko, for suggesting the problem, for the constant support, invaluable guidance throughout this work and fruitful discussions. The author is also grateful to Andrei Dukov for carefully reading the manuscript and providing helpful suggestions and corrections.

\end{document}